\documentclass[11pt, dvipsnames]{amsart}
\usepackage[utf8]{inputenc}
\usepackage{graphicx}
\usepackage{a4wide}
\usepackage{color}
\usepackage{xcolor}
\usepackage{amsmath,amssymb}
\usepackage{float}
\usepackage{soul}
\usepackage{lipsum}
\usepackage{hyperref}
\usepackage{cleveref}

\usepackage[foot]{amsaddr}
\usepackage{float}
\usepackage{amsthm}
\newtheorem{theorem}{Theorem}[section]
\newtheorem{lemma}[theorem]{Lemma}

\newtheorem{remark}{Remark}[section]
\newtheorem{corollary}{Corollary}[section]

\numberwithin{equation}{section} 
\title{Discontinuous Galerkin methods for 3D--1D systems  }
\author{Rami Masri$^1$, Miroslav Kuchta$^1$, and  Beatrice Riviere$^2$}

\email{rami@simula.no, miroslav@simula.no, riviere@rice.edu}
\address{$^1$Department of Numerical Analysis and Scientific Computing, Simula Research Laboratory, Oslo, 0164 Norway. RM gratefully acknowledges support from the Research Council of Norway (RCN) via FRIPRO grant agreement 324239 (EMIx).
MK gratefully acknowledges support from the RCN grant 30336.
}
\address{$^2$Department of Computational Applied Mathematics and Operations Research, Rice University, Houston, TX 77005, USA. BR is partially funded by NSF-DMS 2111459.}
\date{ \today }

\newcommand{\DG}{\mathrm{DG}}
\newcommand{\meshg}{\mathcal{T}_{\Omega}^h}
\newcommand{\meshl}{\mathcal{T}_{\Lambda}^h}
\newcommand{\meshb}{\mathcal{T}_{B}^h}

\newcommand{\bm}[1]{\boldsymbol{#1}}

\newcommand{\dgg}{\mathbb{V}_h^{\Omega}} 
\newcommand{\dgl}{\mathbb{V}_h^{\Lambda}} 
\newcommand{\cgg}{\mathbb{V}_{h,c}^{\Omega}} 
\newcommand{\cgl}{\mathbb{V}_{h,c}^{\Lambda}} 
 
\newcommand{\ds}{\,\mathrm{d}s}

\newcommand{\R}{\mathbb{R}}

\newcommand{\diam}{\mathrm{diam}}

\newcommand{\Bk}{\color{black}}

\begin{document}

\maketitle

\begin{abstract}
  We propose and analyze discontinuous Galerkin (dG)  approximations to 3D-1D coupled systems which model diffusion in a 3D domain containing a small inclusion reduced to  its 1D centerline.  Convergence to weak solutions of a steady state problem is established via deriving a posteriori error estimates and bounds on residuals defined with suitable lift operators. For the time dependent problem, a backward Euler dG formulation is also presented and analysed. Further, we propose a dG method for networks embedded in 3D domains, which is, up to jump terms, locally mass conservative on bifurcation points. Numerical examples in idealized geometries portray our theoretical findings, and simulations in realistic 1D networks show the robustness of our method.

   \vspace{1em}

\end{abstract}

\section{Introduction}

Modeling physiological processes involving the flow and transport within a complex network of vessel-like structures embedded in a 3D domain is crucial. Examples of such processes include drug transport in vascularized tissue \cite{rohan2018modeling,cattaneo2014computational} and solute clearance through the lymphatic vessels in the body \cite{possenti2019numerical} and through the glymphatic system of the brain \cite{masri2023modelling,vinje2021brain}. This modeling setup has applications not only in physiology but also spans  areas such as geosciences \cite{gjerde2020singularity,malenica2018groundwater}.  

To account for a complex network of vessels that typically have a small diameter compared to a surrounding domain, topological model reduction techniques have been proposed \cite{d2008coupling,laurino2019derivation}. Such models reduce the equations posed in 3D vessels to 1D equations posed on their centerlines. Further, these 1D equations are suitably coupled to extended 3D equations in the surrounding. Thereby, 3D--1D models reduce the computational cost while providing a reliable approximation to the full dimensional system. Bounds on the modeling error induced by such a derivation in terms of the vessel diameter are derived for the time dependent convection diffusion 3D-1D problem in \cite{masri2023modelling}, for the steady state diffusion 3D-1D problem in \cite{laurino2019derivation}, and for the steady state 2D-0D problem in \cite{koppl2018mathematical}. 
 
We remark that the 3D-1D model derived in \cite{laurino2019derivation} and further extended in \cite{masri2023modelling} naturally uses the lateral average as a way to restrict 3D functions to  1D inclusions. This differs from the models presented in \cite{d2007multiscale,d2008coupling}, where 1D traces of 3D functions are used; therefore, the functional setting involves special weighted Hilbert spaces. The latter models can be generally viewed as elliptic problems with Dirac line sources. Several finite element schemes have been proposed and analyzed for this class of problems. In addition to the continuous Galerkin method introduced in \cite{d2012finite} and further analyzed in \cite{koppl2014optimal,drelichman2020weighted}, we mention the singularity removal method \cite{gjerde2020singularity}, the mixed approach \cite{gjerde2021mixed}, the interior penalty dG method \cite{masri2023discontinuous} or the Lagrange multiplier approach \cite{doi:10.1137/20M1329664}. It is worth noting that the  papers \cite{gjerde2020singularity, drelichman2020weighted,gjerde2021mixed,masri2023discontinuous} only analyze the 3D problem and assume a given 1D source term. 

 For the 3D-1D problem where the restriction operator is realized via lateral averages,  the continuous Galerkin method is analyzed in \cite{laurino2019derivation}, providing error estimates in  energy norms. To the best of our knowledge, a discontinuous Galerkin method for the coupled 3D--1D system and its analysis are novel. DG approximations have several favorable features such as the  local mass conservation property \cite[Section 2.7.3]{riviere2008discontinuous}. In addition, with dG approximations, local mesh refinement and local high order approximation are easily handled since there are no continuity requirements.  The analysis of dG for the coupled 3D--1D problem requires non--standard arguments as the strong consistency of the method cannot be assumed. This stems from the observation that the 3D solution does not belong to $H^{3/2+\eta}(\Omega)$ for any positive $\eta$, the natural Sobolev space for the interior penalty dG bilinear form. 

We now summarize the main contributions of the paper and give an outline of the contents. 
\begin{itemize}
\item We propose an interior penalty dG method for the coupled 3D--1D problem, and we prove convergence to weak solutions. The main result is given in Theorem \ref{thm:main_result}. 
\item We derive error estimates for regular meshes in Corollary \ref{cor:first_err_estimate} and for graded meshes in Corollary \ref{cor:second_err_estimate}. The second estimate shows that if the mesh is resolved near the boundary of the inclusion, then almost optimal error rates are recovered. 
\item We analyze a backward Euler dG discretization for the time dependent problem by introducing a suitable interpolant which is based on the elliptic projection. The main result is in Theorem \ref{thm:time_dep_estimate}. 
\item For vessel networks embedded in a 3D domain, we propose a dG method with a hybridization technique on bifurcation points. Up to jump terms, this method preserves conservation of mass on such junctions, see Section \ref{sec:network}. We show the well--posedness of this dG formulation. 
\end{itemize}
The rest of this article is organized as follows. Sections \ref{sec:model_pb} and \ref{sec:dG} introduce the model problem and the dG approximation respectively. The error analysis for the steady state problem is included in Section \ref{sec:error}. We analyze a backward Euler dG method for the transient 3D--1D model in Section \ref{sec:time_dependent}. The case of a vessel network inside a 3D domain is studied in Section \ref{sec:network}. In Section \ref{sec:numerics}, we provide numerical examples for manufactured solutions in a 3D--1D setting, for 1D vessel networks, and for realistic 1D networks in 3D tissue.  Conclusions follow in Section \ref{sec:conclusion}.  
\section{Model problem }\label{sec:model_pb}
\subsection{Notation} Given an open domain $O \subset \R^d$, $d \in \{1,2,3\}$, the usual $L^2$ inner product is denoted by  $(f,g)_{O}$ for given real functions $f$ and $g$.  Let  $L^2(O)$ be the Hilbert space with inner product $(\cdot, \cdot)_{O}$ and the usual induced norm $\|\cdot\|_{L^2(O)}$. We  drop the subscript when $O= \Omega$ and let $\|\cdot \|_{L^2(\Omega)} = \|\cdot\|$ and $(f,g)_\Omega  = (f,g)$. Recall the notation of the standard Sobolev spaces $W^{m,p}(O)$ and $H^{m}(O) = W^{m,2}(O)$ for $ m \in \mathbb{N}$ and $ 1 \leq p \leq \infty$. For a given weight $w \in L^{\infty}(O)$ and $w > 0$ a.e. in $O$, the weighted $L^2$ inner product is given  by $(f,g)_{L^2_\omega(O)} = (f,wg)_O$ with  the respective weighted $L_\omega^2$ space:
\begin{equation} \label{eq:weighted_inner_product}
\|f\|_{L^2_w(O)} =\|w^{1/2}f\|_{L^2(O)}, \,\,  L^2_{w}(O) = \{f:O \rightarrow \R \,\,  \vert \,\, \|f\|_{L_w^2(O)} < \infty\}. 
\end{equation}
Similarly, the weighted Hilbert space $H^1_w(O)$ is given by   
\begin{equation}
    H^1_{w}(O) = \{f \in L^2_{w}(O) \,\, \vert \,\,  \|\nabla f\|_{L^2_w(O)} <  \infty\}, 
\end{equation}
where the weighted inner product and norm are
\begin{equation}
(f,g)_{H^1_{w}(O)} = (f,g)_{L^2_w(O)} + (\nabla f, \nabla g)_{L^2_{w}(O)}, \,\,\,  \|f\|_{H^1_{w}(O)}^2 = \|f\|_{L^2_w(O)}^2 + \|\nabla f \|_{L^2_w(O)}^2. 
\end{equation}
We omit the subscript/weight $w$ when $w = 1$. 
Throughout the paper, we denote by $C$ a generic constant independent of mesh parameters. We use the standard notation $A \lesssim B$ for $A \leq C \, B$ and  $A \approx B$ for $A \leq C \, B$ and $B \leq C \, A$. 
\subsection{The 3D-1D model}
Let $\Omega \subset \mathbb{R}^3$ be a bounded domain with a one dimensional inclusion $\Lambda$. We assume that $\Lambda$ is parametrized  by $\bm \lambda(s), s \in [0,L]$ and is strictly included in $\Omega$, $\bm \lambda$ is $C^2$ regular, and (for simplicity) $\|\bm \lambda' (s) \| = 1$ so that the arc length and $s$ coincide.  We further define $B_{\Lambda}$ as a generalized cylinder with centerline $\Lambda$. The boundary of $B_{\Lambda}$ will be denoted by $\Gamma$. See Figure \ref{fig:domains} for an illustration of the considered geometry.
\begin{figure}[ht]
    \centering
    \includegraphics[angle=-90,scale = 0.7]{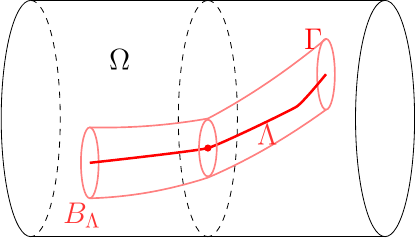}
    \caption{Illustration of the considered geometry}
    \label{fig:domains}
\end{figure}
A cross-section of $B_{\Lambda}$ at $s\in [0,L]$ is denoted by $\Theta(s)$ with area $A(s)$ and perimeter $P(s)$.  We assume that there are positive constants $a_0, a_1$ such that $a_0\leq A(s)+P(s) \leq a_1$ for all $s$.  We also assume that $A$ belongs to $\mathcal{C}^1([0,L])$. For a  function $u \in L^1(\partial \Theta(s))$, we define the lateral average $\bar{u}$ as 
\begin{equation}
  \overline{u}(s)= \frac{1}{P(s)}  \int_{\partial \Theta(s)} u.  \label{eq:average_operator}
\end{equation} 
The 3D-1D model that we  consider results from a reduction of a 3D-3D model in $\Omega \backslash B_{\Lambda}$ and in $B_{\Lambda}$ with Robin type interface conditions on $\Gamma$. This condition models the membrane $\Gamma$ as semi-permeable with permeability constant $\xi > 0$. Averaging the equations in $B_{\Lambda}$ and formally extending the equations in $\Omega \backslash B_{\Lambda}$, one obtains the following coupled system
\begin{alignat}{2}
- \Delta u + \xi (\overline{u} - \hat u ) \delta_{\Gamma} &= f, && \,\,\,  \mathrm{in} \,\,\, \Omega, \label{eq:pde_1}  \\
-\mathrm d_s (A \, \mathrm{d}_s \hat u ) + P \, \xi (\hat u -\overline{u}) & = A \, \hat f, && \,\,\, \mathrm{in} \,\,\, \Lambda. \label{eq:pde_2} 
\end{alignat}
We refer to \cite{laurino2019derivation,masri2023modelling} for  details on the derivation and on the model error analysis.  
The source terms $f \in L^2(\Omega)$ and $\hat f \in L^2_A(\Lambda)$ are given. The above equations are to be understood in the weak sense, and the functional $(\overline{u} - \hat u) \delta_{\Gamma}$ is defined over $H^1(\Omega)$ as
\begin{equation}
(\overline{u} - \hat u ) \delta_{\Gamma}(v) 
=  \int_{\Lambda} P (\overline{u} - \hat u ) \overline v , \quad \forall v \in H^1(\Omega). 
\end{equation}
The above functional is well-defined since an application of  Cauchy-Schwarz's inequality and trace theorem yields
\begin{equation}
\|\overline{v}\|_{L^2_P(\Lambda)} \leq \|v\|_{L^2(\Gamma)} \lesssim \|v\|_{H^1(\Omega)}, \quad \forall v \in H^1(\Omega). \label{eq:trace_estimate_avg}
\end{equation}
The system \eqref{eq:pde_1}--\eqref{eq:pde_2} is complemented by the following boundary conditions.
\begin{equation}
 u  = 0 ,  \,\,\, \mathrm{on} \,\, \partial \Omega,   \quad \mathrm{and} \quad 
  A \, \mathrm{d}_s \hat u  = 0,  \,\,\, \mathrm{on} \,\, s =\{0,L\}. 
\end{equation}
To introduce the weak formulation of \eqref{eq:pde_1}--\eqref{eq:pde_2}, we define the following bilinear forms. 
\begin{alignat}{2}
  a(u,v)  &= ( \nabla u ,  \nabla v ), && \quad \forall u, v \in H^1(\Omega), \\ 
  a_{\Lambda} (\hat u, \hat v)  &=  ( \mathrm{d}_s \hat u,  \mathrm d_s  \hat v)_{L^2_A(\Lambda)} ,&& \quad \forall \hat u, \hat v \in H_{A}^1(\Lambda), \\ 
  b_{\Lambda} ( \hat v,\hat w ) &=   (\xi  \hat v, \hat w  )_{L^2_P(\Lambda)},  && \quad  \forall  \hat v, \hat w  \in L^2_{P}(\Lambda). 
\end{alignat}
The weak formulation of the coupled 3D-1D problem then reads \cite{laurino2019derivation}: Find $\bm{u} = (u,\hat u) \in H_0^1(\Omega)\times H_A^1(\Lambda)$ such that 
\begin{alignat}{2}
a(u,v) + b_{\Lambda}(\overline{u} - \hat u , \overline{ v}) &= (f,v)_{\Omega}, && \quad \forall  v \in H_0^1(\Omega),   \\
a_{\Lambda} ( \hat u, \hat v ) + b_{\Lambda}( \hat u - \overline u , \hat v ) &= (\hat f, \hat v)_{L^2_A(\Lambda)}, && \quad \forall  \hat v \in H^1_{A}(\Lambda). 
\end{alignat}
The terms given in $b_{\Lambda}$ model the coupling between the 3D solution $u$ and the 1D solution $\hat u$. 
Equivalently, one can write the above as follows.  Find $\bm{u} =  (u,\hat u) \in H_0^1(\Omega)\times H_A^1(\Lambda)$ such that 
\begin{equation}
\mathcal A(\bm{u},\bm{v}) = (f,v)_{\Omega} +(\hat f, \hat v)_{L^2_A(\Lambda)} , \quad \forall \bm{v}= (v, \hat v) \in  H_0^1(\Omega)\times H_A^1(\Lambda), \label{eq:weak_form_grouped} 
\end{equation}
where we define for $\bm{u} = (u,\hat u ), \bm{v} = (v ,\hat v) \in H_0^1(\Omega) \times H^1_{A}(\Lambda)$ 
\begin{equation} 
  \mathcal{A} (\bm{u}, \bm{v}) = a(u,v) + a_{\Lambda}(\hat u , \hat v) + b_{\Lambda} (\overline{u} - \hat{u}, \overline{v} - \hat v ). 
\end{equation}
The problem  given in \eqref{eq:weak_form_grouped} is well-posed, see  \cite[Section 3.2]{laurino2019derivation}. 
\section{Discontinuous Galerkin formulation}\label{sec:dG}
\subsection{Meshes and DG spaces}We consider a family of regular partitions of $\Omega$  made of tetrahedra and  denoted by $\mathcal{T}_{\Omega}^h$. The mesh-size is $h = \max_{K \in \meshg} h_K $, where $h_K = \mathrm{diam}(K)$. We associate with  $\mathcal{T}_{\Omega}^h$, the space $H^1(\meshg)$ of broken $H^1$ functions in $\Omega$, and  a finite dimensional space $\mathbb{V}_h^{\Omega}$  of broken piecewise polynomials of order $k_1$. 
\begin{equation}
    \mathbb{V}_h^{\Omega} = \{v_h \in L^2(\Omega), \,\,\,  v_h \in \mathbb{P}_{k_1}(K), \,\,\, \forall K \in \mathcal{T}_{\Omega}^h\}. 
\end{equation}
Similarly, we let $\mathcal{T}_{\Lambda}^h   = \{ (s_{i-1}, s_{i}), \,\, i = 1,\ldots, N\} $ with $s_{0}=0$ and $s_{N} = L$ be a family of  uniform  partitions of $\Lambda$,  with mesh size $h_\Lambda = s_i - s_{i-1}$. We let $H^1(\meshl)$ be the space of broken $H^1$ functions in $\Lambda$, and we let  $\mathbb{V}^\Lambda_h$ be the respective space of broken piecewise polynomials of order $k_2$.  
\begin{equation}
\label{eq:dg_space_1D}
    \mathbb{V}_h^{\Lambda} = \{ \hat v_h \in L^2((0,L)), \,\,\,  \hat v_h \in \mathbb{P}_{k_2}((s_{i-1},s_i)), \,\,\, 1\leq i \leq N\}.  
\end{equation}
For each $1\leq i\leq N$,  we let $B_i$ be the portion of $B_{\Lambda}$ obtained when $s$ is restricted to $\Lambda_i = (s_{i-1},s_i)$. That is, we have that $$\bigcup_{1\leq i\leq N} B_i  = B_{\Lambda}. $$
For each $1\leq i\leq N$, we now define neighborhoods of $[s_{i-1},s_i]$ consisting of 3D elements in $\meshg$. Namely, we define 
\begin{equation}
    \omega_{i} = \{K \in \meshg, \,\,\, K \cap \partial B_i \neq \emptyset\}.  \label{eq:def_w_EL}
\end{equation}
We can then write 
\begin{equation}
    \meshb := \{K \in \meshg, \,\,\, K \cap \partial B_{\Lambda} \neq \emptyset\} = \bigcup_{1\leq i\leq N} \omega_i.  \label{eq:mesh_b}
\end{equation}

We assume that if $K \cap B_{\Lambda} \neq \emptyset$, then  the two-dimensional Lebesgue measure of $\partial K \cap \partial B_{\Lambda}$ is zero. This ensures that the average operator $\overline \cdot $ given in \eqref{eq:average_operator} is well defined over $H^1(\meshg)$.  

Further, 
we assume the following relation between the level of refinement for the 3D domain and that of the 1D domain. For $K \in \meshb$, let $\mathcal{I}_K$ be the set of integers $i_0$  such that $K \in \omega_{i_0}$. We assume that the cardinality of $\mathcal{I}_K$ is bounded above by a constant independent of $K$ and of $h$.
Essentially, this assumption means the 3D mesh intersecting $\partial B_{\Lambda}$ can not remain fixed while the 1D mesh is refined.

We also denote by $\Gamma_h$ the set of all interior faces in $\mathcal{T}_{\Omega}^h$. The set of edges belonging to elements in $\meshb$  is decomposed by defining \begin{equation}
    \Gamma_{i}  = \{F  \in \Gamma_h, \,\,\, F \subset \partial K, \,\,\, \mathrm{where } \,\,\, K \in \omega_i \}. \label{eq:def_Gamma_i}
\end{equation} 

For each interior face $F$, we associate a unit normal
vector $\bm{n}_F$ and we denote by $K_F^1$ and $K_F^2$ the two elements that share $F$ such that the vector $\bm{n}_F$ points from 
$K_F^1$ to $K_F^2$. We denote the average and the jump of a function $v_h \in \mathbb V_h^\Omega$ by $\{v_h\}$
and $[v_h]$ respectively.  
\begin{align}
\{v_h\} = \frac{1}{2} \left( v_h\vert_{K_F^1} + v_h \vert_{K_F^2}  \right), \quad [v_h] = v_h \vert_{K_F^1} - v_h\vert_{K_F^2},
\quad \forall F \in\Gamma_h.  
\end{align} 
If $F = \partial \Omega \cap \partial K_F^1$, then the average and the jump are given by
\begin{equation}
    [v] = \{ v\} = v\vert_{K_F^1}. 
\end{equation}
The area of $F \in \Gamma_h \cup \partial \Omega$ is denoted by $|F|$. Similar definitions are adopted for jumps and averages of $\hat{v}_h \in \mathbb{V}_h^{\Lambda}$ on the nodes $s_i$. 
\begin{alignat*}{2}   
1\leq i\leq N-1, & \quad [\hat v_h]_{s_i} && = 
\lim_{t\rightarrow 0, t>0} \hat v_h(s_i  - t)
-\lim_{t\rightarrow 0, t>0} \hat v_h(s_i +  t), \\ 
1\leq i\leq N-1,&  \quad \{\hat v_h\}_{s_i} && = 
\frac12 \lim_{t\rightarrow 0, t>0} \hat v_h(s_i+t)
+\frac12 \lim_{t\rightarrow 0, t>0} \hat v_h(s_i-t), \\
  & \,\,\,\,\,\,\,  [\hat{v}_h]_{s_0} && = - \hat{v}_h(s_0), \quad
[\hat{v}_h]_{s_N} = \hat{v}_h(s_N).
\end{alignat*}
For $v \in H^1(\meshg)$, we define the norm for $\sigma_{\Omega} > 0$
\begin{equation}
    \|v\|_{\dgg}^2 = \sum_{K \in \meshg} \|\nabla v\|_{L^2(K)}^2 + \sum_{F \in\Gamma_h \cup \partial \Omega} \frac{ \sigma_\Omega}{|F|^{1/2}}\|[v]\|^2_{L^2(F)}.  
\end{equation}
A Poincar\'e inequality holds in   $H^1(\meshg)$ (see for e.g \cite[Remark 4.15]{di2011mathematical}): 
\begin{equation}
    \|v\|_{L^2(\Omega)} \lesssim \|v\|_{\dgg}, \quad \forall v \in H^1(\meshg).    \label{eq:Poincare}
\end{equation}
For  $\hat{v} \in H^1(\meshl)$, we define the semi-norm for $\sigma_{\Lambda} > 0$
\begin{equation}
    |\hat{v}|_{\dgl}^2 = \sum_{i=1}^N \|\mathrm{d}_s \hat v\|_{L^2((s_{i-1},s_i))}^2 + \sum_{i=1}^{N-1} \frac{\sigma_\Lambda}{h_{\Lambda}} [\hat v]_{s_i}^2. \label{eq:def_dg_e}
\end{equation}
The above definitions allow us to introduce the norm $\|\cdot \|_{\mathrm{DG}}$ on $H^1(\meshg) \times H^1(\meshl)$. For $\bm{v} = (v, \hat v)$, 
\begin{equation} \label{eq:norm_dg_full}
\|\bm{v}\|_{\mathrm{DG}}^2 =  \|v\|_{\dgg}^2  +  |\hat{v}|_{\dgl}^2 + \|\overline{v} - \hat{v} \|^2_{L_P^2(\Lambda)}. 
\end{equation}
The above indeed defines a norm. 
If $\|\bm{v}\|_{\DG} = 0$, then $v = 0 $ since $\|\cdot \|_{\dgg}$ is a norm on $H^1(\meshg)$. This implies that $ \|\hat v\|_{L_P^2(\Lambda)} = 0 $ and $\hat v= 0.$
\subsection{The numerical method} We use interior penalty discontinuous Galerkin  forms \cite{riviere2008discontinuous}. Define  $a_{h}(\cdot, \cdot): \dgg \times \dgg \rightarrow \mathbb{R}$: 
\begin{align}\label{eq:dg_bilinear_form}  
a_{h}(u,v) = \sum_{K \in \meshg} \int_K \nabla u \cdot  \nabla v   - \sum_{ F   \in \Gamma_h \cup \partial \Omega} \int_F  \{\nabla u\}\cdot \bm{n}_F [v]   \\ + \epsilon_1 \sum_{F \in \Gamma_h \cup \partial \Omega} \int_F  \{\nabla v\}\cdot \bm{n}_F [u]  + \sum_{F \in \Gamma_h \cup \partial \Omega}  \frac{\sigma_\Omega}{|F|^{1/2}} \int_F [u][v].  \nonumber 
\end{align} 
For the 1D discrete solution, we introduce the form $ a_{\Lambda,h}(\cdot,\cdot): \dgl \times \dgl \rightarrow \mathbb{R} $
\begin{align}
    a_{\Lambda,h}(\hat u, \hat v) = &\sum_{i=1}^{N} \int_{s_{i-1}}^{s_i} A \frac{\mathrm d\hat u}{\ds} \frac{\mathrm d\hat v}{\ds} - \sum_{i=1}^{N - 1}  \left \{A\frac{\mathrm d\hat u}{\ds}\right \}_{s_i} [\hat v]_{s_i} \\ & + \epsilon_2   \sum_{i=1}^{N- 1}  \left \{A\frac{\mathrm d\hat v}{\ds}\right \}_{s_i} [\hat u]_{s_i} +  \sum_{i=1}^{N - 1}  \frac{\sigma_{\Lambda}}{h_\Lambda}  [\hat u]_{s_i} [\hat v]_{s_i}.  \nonumber
\end{align}
In the above, $\epsilon_1, \epsilon_2 \in \{ -1,0,1\}$ lead to 
symmetric, incomplete, or non-symmetric discretizations, and $\sigma_\Lambda, \sigma_\Omega > 0$ are penalty parameters. The dG formulation of problem \eqref{eq:weak_form_grouped} then reads as follows. Find $\bm{u}_h = (u_h, \hat{u}_h) \in \dgg \times \dgl$ such that 
\begin{align}
\mathcal{A}_h (\bm{u}_h, \bm{v}_h) = (f,v_h)_{\Omega} + (\hat f , \hat v_h)_{\Lambda}, \quad \forall \bm{v}_h = (v_h, \hat{v}_h) \in  \dgg \times \dgl,  \label{eq:discrete_form_dg}
\end{align}
where we defined the form $\mathcal{A}_h (\cdot, \cdot): (\dgg \times \dgl)^2  \rightarrow \mathbb R$: 
\begin{align} \label{eq:dg_form_combined}
    \mathcal{A}_h (\bm{u}_h, \bm{v}_h) = a_h(u_h, v_h) + a_{\Lambda, h} (\hat{u}_h, \hat{v}_h) +  b_{\Lambda} (\overline{u}_h - \hat{u}_h, \overline{v}_h - \hat{v}_h). 
\end{align}
It is important to note that the interface $\Gamma$ does not need to be resolved by the mesh to realize the coupling term $b_{\Lambda}$; identifying the elements intersecting $\Gamma$ is sufficient. To show the well--posedness of the discrete dG formulation, we first show the coercivity of $\mathcal{A}_h$ with respect to the norm defined in \eqref{eq:norm_dg_full}.
\begin{lemma}[Coercivity] \label{lemma:coercivity}
For suitably chosen penalty parameters $\sigma_\Lambda$ and $\sigma_{\Omega}$, there exists a constant $C_{\mathrm{coerc}}$ such that 
\begin{equation} \label{eq:coercivity}
\mathcal{A}_h(\bm{u}_h, \bm{u}_h) \geq C_{\mathrm{coerc}} \|\bm{u}_h\|_{\DG}^2, \quad \forall \bm u_h \in \dgg \times \dgl . 
\end{equation}
\end{lemma}
\begin{proof}
From standard analysis of dG methods, if $\sigma_\Omega$ is large enough whenever $\epsilon_1 = -1$ or  for any $\sigma_\Omega$ when $\epsilon_1 = 1$ (same conditions apply for $\sigma_\Lambda$ and $\epsilon_2$), we have 
\begin{align}
 a_h (u_h, u_h) \geq C_1 \|u_h\|_{\dgg}^2,  \quad a_{\Lambda,h}(\hat u_h ,\hat u_h) \geq C_2 |\hat{u}_h|_{\dgl}^2.  \label{eq:coercivity_local}
\end{align}
The result then immediately follows from the above and from the definition of $\mathcal{A}_h(\cdot, \cdot)$.
\end{proof}
\begin{lemma}[Existence and uniqueness of solutions] There exists a unique pair $(u_h, \hat{u}_h) \in \dgg\times \dgl$ solving \eqref{eq:discrete_form_dg}. \label{lemma:well_posedness_1}
\end{lemma}
\begin{proof}
From the coercivity property, it easily follows that the solution is unique. Since this is a square linear system in finite dimensions, existence follows.  
\end{proof}
\section{Error analysis}\label{sec:error}
The main difficulty in the error analysis of the dG formulation is that the strong consistency of the method can not be assumed. Indeed, under sufficient regularity assumptions on the domain, one can only show that  the 3D solution  $u$  of \eqref{eq:weak_form_grouped} belongs to $H^{3/2-\eta}(\Omega)$ for $\eta >0$ \cite{laurino2019derivation}.  However, the  form $a_h$ can not be extended to this space  since the traces of gradients for functions in $H^{3/2-\eta}(\Omega)$ are not well defined. Therefore, we adopt here a combination of a priori and a posteriori error estimates within the framework proposed by Gudi \cite{gudi2010new} to prove convergence. 
The main result is provided in Theorem \ref{thm:main_result}.
\subsection{Preliminary lemmas}We first introduce the conforming spaces  $\mathbb{V}_{h,c}^\Omega \subset H^1_0(\Omega)$ of continuous piecewise linear functions defined over $\meshg$ in $\Omega$.  Similarly, we let $\mathbb{V}_{h,c}^\Lambda \subset H^1(\Lambda)$ be the respective space of continuous piecewise linear functions defined over $\meshl$ in $\Lambda$.  
\begin{lemma} \label{lemma:enrich}
    There exists an enriching map $\bm{E} =(E,\hat{E}): \dgg \times \dgl \rightarrow \cgg \times \cgl$
such that 
\begin{align}
|E v|_{H^1(\Omega)} \lesssim \|v \|_{\dgg}, &\quad |\hat{E} \hat{v} |_{H^1(\Lambda)} \lesssim | \hat{v}  |_{\dgl}, \label{eq:stability_E}\\ 
\left(\sum_{K \in \meshg} h_K^{-2} \|E v   - v\|^2_{L^2(K)} \right)^{1/2} \lesssim \|v\|_{\dgg}, &\quad \|\hat{E} \hat{v} - \hat v \|_{L^2(\Lambda)} \lesssim h_{\Lambda}  | \hat{v}  |_{\dgl}.  \label{eq:approx_E}
\end{align}
\end{lemma}
\begin{proof}An enriching map with the above properties can be constructed as a nodal Lagrange interpolant with nodal values taken as averages  of $v$ ($\hat v$), see  \cite[Theorem 2.2] {karakashian2003posteriori} and 
\cite[Section 5.5.2]{di2011mathematical}. 
Another approach is to apply a Scott-Zhang interpolant to a Crouzeix-Raviart 
 correction, see  \cite[Lemma 6.2]{girault2016strong}.   
\end{proof} 
We now define $L^2$--projections.  Let $K \in \meshg$ and  $\Lambda_i = (s_{i-1},s_i)$ for $ i \in \{1,\ldots,N\}$. For any $(w, \hat w) \in L^2(K) \times L^2(\Lambda_i)$, define $(\pi_h w, \hat \pi_h \hat w) \in \mathbb{P}^{k_1}(K) \times \mathbb{P}^{k_2}(\Lambda_i)$ such that 
\begin{equation} \forall v_h \in \mathbb P^{k_1}(K), \quad 
(\pi_h w - w, v_h)_K = 0, \quad \forall \hat{v}_h \in \mathbb P^{k_2} (\Lambda_i),  \quad (\hat\pi_h \hat w - \hat w, \hat{v}_h)_{L_P^2(\Lambda_i)} = 0.  \label{eq:local_l2_projection}
\end{equation} 
\begin{lemma}[Properties of the $L^2$--projections]  Let $s \in \{0,\ldots, k+1\}$, $m \in \{0,\ldots,s\}$, $K \in \meshg$, and $\Lambda_i = (s_{i-1},s_i)$ for $i  \in \{1,\ldots, N\}$. Assume that $w \in H^{s}(K)$ and $\hat w \in H^{s}(\Lambda_i)$. Then,
\begin{align}
\|\pi_h w - w\|_{H^m(K)} \lesssim h_K^{s-m} \| w \|_{H^{s}(K)},  \quad \|\hat \pi_h \hat w - \hat w\|_{H^m(\Lambda_i)} \lesssim h_{\Lambda}^{s-m} \| \hat w \|_{H^{s}(\Lambda_i)}.  \label{eq:approximation_l2_projection}
\end{align}
   In addition, the $L^2$ projection is stable in the dG norm. Namely, 
\begin{equation}
    \|\pi_h w \|_{\dgg} \lesssim \|w\|_{H^1(\meshg)}, \quad \forall w \in H^1(\meshg).  \label{eq:stability_l2_projection}
\end{equation}
\end{lemma}
\begin{proof}
Proofs of the estimates in  \eqref{eq:approximation_l2_projection} can be found in \cite[Lemma 1.58]{di2011mathematical}.  The proof of \eqref{eq:stability_l2_projection} follows from applications of trace estimates and \eqref{eq:approximation_l2_projection}. 
\end{proof}
We now state a local trace inequality on $\partial B_{\Lambda}$. The proof of the following estimate is due to Wu and Xiao \cite[Lemma 3.1]{wu2019unfitted}. 
\begin{lemma}[Local trace estimate on $\partial B_{\Lambda}$] There exists a constant $h_0$ such that for all $h \leq h_0$ and $K \in \meshb$,  the following estimates hold 
\begin{alignat}{2}
     \|v\|_{L^2(\partial B_{\Lambda} \cap  \overline{K} )} &\lesssim h_{K}^{ -1/2} \|v\|_{L^2(K)} + h_{K}^{1/2} \|\nabla v\|_{L^2(K)},&& \quad \forall v \in H^1(K), \label{eq:trace_cont_lambda} \\ 
       \|v_h\|_{L^2(\partial B_{\Lambda} \cap  \overline{K})} &\lesssim h_{K}^{-1/2} \|v_h\|_{L^2(K)},  && \quad \forall v_h \in \mathbb{P}_k(K), \quad \forall k\geq 1.  \label{eq:trace_disc_lambda}
\end{alignat}
\end{lemma}
Hereinafter, we assume that $h \leq h_0$. We now show a global trace inequality. 
\begin{lemma}[Trace estimate] \label{lemma:trace_for_broken}
For $u \in H^1(\meshg)$,  there holds 
\begin{equation} \|\overline{u}\|_{L^2_P(\Lambda)} \lesssim \|u\|_{\dgg}. 
    \label{eq:boundubar}
 \end{equation}
\end{lemma}
\begin{proof}
We start by showing the result for $v_h \in \mathbb{V}_h^{\Omega}$.  
Let $\Lambda_i = (s_{i-1},s_i)$. We use triangle and Cauchy--Schwarz's inequalities to obtain that for any $1\leq i\leq N$, 
    \begin{multline}
         \|\overline{v_h} \|_{L^2_P(\Lambda_i)} \leq \|\overline{v_h}  - \overline{E v_h}\|_{L^2_P(\Lambda_i)} + \|\overline{E v_h}\|_{L^2_P(\Lambda_i)}\leq \|v_h -Ev_h \|_{L^2(\partial B_i)} + \|\overline{E v_h}\|_{L^2_P(\Lambda_i)}.   \label{eq:global_trace_0_red}  
    \end{multline} 
Recall the definition of $\omega_i$ in \eqref{eq:def_w_EL} and note that 
 \begin{align}
   \|v_h -Ev_h \|_{L^2(\partial B_i)}^2 = \sum_{K\in \omega_i} \|v_h -Ev_h \|^2_{L^2(\partial B_i \cap  \bar{K}  )}.
 \end{align}
 With \eqref{eq:trace_disc_lambda},  we obtain that 
 \[
        \|v_h -Ev_h \|_{L^2(\partial B_i)}^2 \lesssim \sum_{K \in \omega_i}  h_{K}^{-1} \|v_h -Ev_h \|_{L^2(K)}^2. 
 \]
 Summing over $i$  and using the global bound \eqref{eq:approx_E} yield 
 \[
 \sum_{i=1}^N \|v_h -Ev_h \|_{L^2(\partial B_i)}^2 
 \lesssim \sum_{K\in\mathcal{T}_B^h} h_{K}^{-1} \|v_h -Ev_h \|_{L^2(K)}^2 
 \lesssim h \Vert v_h \Vert_{\dgg}^2.
 \]
 Therefore, with \eqref{eq:trace_estimate_avg}, we obtain 
 \begin{align}
\|\overline{v_h}\|_{L^2_P(\Lambda)}^2 =
 \sum_{i=1}^N \|\overline{v_h} \|_{L^2_P(\Lambda_i)}^2
 \lesssim h \Vert v_h \Vert_{\dgg}^2
 + \Vert \overline{Ev_h}\Vert_{L^2_P(\Lambda)}^2
 \lesssim h \Vert v_h \Vert_{\dgg}^2 + \Vert Ev_h \Vert_{H^1(\Omega)}^2.  \label{eq:trace_0}
 \end{align}
 With Poincar\'e's inequality \eqref{eq:Poincare}, and the properties of $E$ \eqref{eq:stability_E}--\eqref{eq:approx_E}, we obtain the bound 
 \[\Vert Ev_h \Vert_{H^1(\Omega)}^2  \leq  \|Ev_h - v_h\|_{L^2(\Omega)}^2 + \|v_h\|^2_{L^2(\Omega)} + |Ev_h|_{H^1(\Omega)}^2 \lesssim \|v_h\|^2_{\dgg}. \]
Substituting the above in \eqref{eq:trace_0} yields \eqref{eq:boundubar} for $v_h \in \mathbb V_h^{\Omega}$. Consider now $u \in H^{1}(\meshg)$ and recall that $\pi_h u$ is the local $L^2$ projection on $\mathbb V_h^\Omega$. Then, by Cauchy--Schwarz's inequality and \eqref{eq:trace_cont_lambda}, we have that 
\begin{multline*}
\|\overline{u} - \overline{\pi_h u }\|^2_{L^2_P(\Lambda_i)} \leq \sum_{K \in \omega_i } \|u - \pi_h u \|^2_{L^2(\partial B_i \cap \overline K)}  \\  \lesssim  \sum_{K \in \omega_i} \left(h_{K}^{-1} \| u - \pi_h u \|^2_{L^2(K)} + h_K\| \nabla (u - \pi_h u) \|^2_{L^2(K)}\right) \lesssim 
\sum_{K \in \omega_i} h_K \|u\|^2_{H^1(K)}. 
\end{multline*}
In the above, we used the properties of the $L^2$ projection given in \eqref{eq:approximation_l2_projection}. 
 Then, using triangle inequality and  \eqref{eq:boundubar} for $\mathbb V_h^{\Omega}$, we obtain 
\begin{align}
\|\overline{u}\|_{L^2_P(\Lambda)} \leq \|\overline{u} - \overline{\pi_h u}\|_{L^2_P(\Lambda)} + \|\overline{\pi_h u}\|_{L^2_P(\Lambda)} \lesssim h^{1/2} (  
\sum_{K \in \meshg} \|u\|^2_{H^1(K)})^{1/2}  + \|\pi_h u \|_{\dgg}.
 \end{align}
 The result is concluded by Poincar\'e's inequality \eqref{eq:Poincare} and the stability of the $L^2$ projection $\pi_h$ in the $\|\cdot \|_{\dgg}$ norm, see \eqref{eq:stability_l2_projection}.   
   \end{proof}
 A consequence of \eqref{eq:boundubar} and the triangle inequality is the following bound:
\begin{equation}\label{eq:bounduhat}
\forall \bm u = (u,\hat u) \in H^1(\meshg)\times H^1(\meshl), \quad
\Vert \hat u\Vert_{L_P^2(\Lambda)} \lesssim \Vert \bm u \Vert_{\DG}.
\end{equation}

We will make use of lift operators. For a given $(u,\hat{u}) \in H^1(\meshg) \times H^1(\meshl)$, define $(L_h u, \hat{L}_h \hat{u}) \in \dgg \times \dgl  $ such that 
\begin{align}
(L_h u, w_h)_{\Omega} + (\hat{L}_h \hat u, \hat w_h)_{L^2_P(\Lambda)} = b_{\Lambda}( \overline{u} - \hat{u}, \overline{w}_h - \hat{w}_h), \quad \forall (w_h, \hat{w}_h) \in \dgg \times \dgl. \label{eq:def_lift_operator}
\end{align}  The existence of $(L_h u, \hat{L}_h \hat{u})$ easily follows from uniqueness. We show the following estimate.  
\begin{lemma}[Lift operator] \label{lemma:lift_operator}Given $(u,\hat{u}) \in H^1(\meshg) \times H^1(\meshl)$, let $(L_h u, \hat{L}_h \hat{u}) \in \dgg \times \dgl$ be defined by  \eqref{eq:def_lift_operator}. There holds 
\begin{align}
\sum_{K \in \meshg} h_K \|L_h u\|^2_{L^2(K)}+ \|\hat{L}_h \hat u\|^2_{L^2_P(\Lambda)} \lesssim \|u\|^2_{\dgg} + \|\hat u\|^2_{L^2_{P}(\Lambda)}. \label{eq:stability_lift_operator}
\end{align}
\end{lemma}
\begin{proof}
Choosing $(w_h, \hat{w}_h) = (0, \hat{L}_h \hat u)$ in \eqref{eq:def_lift_operator} and using Cauchy-Schwarz's inequality and  \eqref{eq:boundubar},   we have 
\begin{align}
\|\hat{L}_h \hat u\|^2_{L^2_P(\Lambda)} \leq  \xi  \|\overline u - \hat u \|_{L^2_P(\Lambda)} \|\hat L_h \hat u\|_{L^2_P(\Lambda)} \lesssim (\|u\|_{\dgg} + \|\hat u \|_{L^2_P(\Lambda)})\|\hat L_h\hat u\|_{L^2_P(\Lambda)} . 
\end{align}
This shows the bound on the second term in \eqref{eq:stability_lift_operator}. Next, fix $K \in \meshb$ and recall that $\mathcal{I}_K$ be the set of integers $i_0$  such that $K \in \omega_{i_0}$ where we assume that the cardinality of $\mathcal{I}_K$ is bounded above by a small constant independent of $K$. 
 In \eqref{eq:def_lift_operator}, choose $ \hat{w}_h= 0 $ and $w_h = (L_h u)\chi_K$ where $\chi_K$ is the characteristic function on $K$.  We obtain
\begin{align} \label{eq:first_exp_lift}
\|L_h u \|_{L^2(K)} ^2 \lesssim 
 \sum_{i_0\in\mathcal{I}_K} (\|\overline{u}\|_{L^2_P((s_{i_0-1} , s_{i_0}))} + \|\hat u \|_{L^2_P((s_{i_0-1} , s_{i_0}))})  \|\overline{w}_h \|_{L^2_P((s_{i_0-1} , s_{i_0}))}. 
\end{align}
 We now use Cauchy--Schwarz's inequality, the observation that $w_h$ is locally supported in $K$, and trace inequality \eqref{eq:trace_disc_lambda}. We estimate  
\begin{align}
\|\overline{w}_h \|_{L^2_P((s_{i_0-1} , s_{i_0}))}  \leq \|w_h\|_{L^2(\partial B_{i_0})}  = \|L_h u \|_{L^2( \partial B_{i_0}  \cap \bar{K} )}
& \lesssim h_{K}^{-1/2} \|L_h u \|_{L^2(K) }. 
\end{align}
Thus, we conclude that 
\begin{equation}
    \|L_h u\|^2_{L^2(K)} \lesssim h_{K}^{-1/2}  \sum_{i_0\in\mathcal{I}_K}(\|\overline{u}\|_{L^2_P((s_{i_0-1} , s_{i_0}))} + \|\hat u \|_{L^2_P((s_{i_0-1} , s_{i_0}))})  \|L_h u\|_{L^2(K)}. 
\end{equation}
Summing the above bound over $K\in \meshb$ and using Cauchy-Schwarz's inequality yield  
\begin{align*}
  \sum_{K \in \meshb} h_{K} \|L_h u \|_{L^2(K)}^2 \lesssim (\|\overline{u}\|_{L^2_P(\Lambda)} + \|\hat{u}\|_{L^2_P(\Lambda)}) \left(  \sum_{K \in \meshb} h_{K} \|L_h u \|_{L^2(K)}^2 \right)^{1/2}. 
\end{align*}
With Lemma \ref{lemma:trace_for_broken} and with noting that $L_h u|_K = 0 $ for $K \notin \meshb$, we conclude the result.
\end{proof}
\subsection{Main result and proof outline}
The main convergence result reads as follows. 
\begin{theorem}\label{thm:main_result}
Let $\bm{u} = (u, \hat{u}) \in H^1_0 (\Omega) \times H^1(\Lambda)$  be the weak solution defined by \eqref{eq:weak_form_grouped}, and let $\bm{u}_h= (u_h, \hat{u}_h) \in \dgg \times \dgl$ be the discrete solution defined by \eqref{eq:discrete_form_dg}.  
Recall that $h_B = \max_{K \in \meshb} h_{K}$. The following estimate holds.  
\begin{multline} \label{eq:main_estimate}
\|\bm u - \bm u_h\|_{\DG} \lesssim \inf_{\bm{v} \in \dgg \times \dgl} \|\bm u - \bm v\|_{\DG}\\
+ h \|f- \pi_h f\|_{ L^2(\Omega) } + h_{\Lambda} \|\hat{f} - \hat \pi_h \hat f \|_{L^2_{A}(\Lambda)} + h_{B}^{1/2}\|\overline{u} - \hat u\|_{L^2_{P}(\Lambda)}. 
\end{multline}
\end{theorem}
\begin{proof}
 Here, we present the main steps of the proof. The details are given    in the next section. We have, see Lemma \ref{lemma:condition_N_2} for the proof, 
     \begin{multline}
 \|\bm{u} - \bm{u}_h\|_{\DG}  \\ \lesssim \inf_{\bm{v} \in \dgg\times \dgl}  \left ( \|\bm{u} - \bm{v}\|_{\DG}  + \sup_{\bm{\phi} \in \dgg\times \dgl } \frac{(f,\phi - E \phi)_{\Omega} + (\hat f, \hat \phi - \hat E \hat \phi)_{L^2_{A}(\Lambda)} - \mathcal{A}_h(\bm{v}, \bm{\phi} - \bm{E} \bm \phi)}{\|\bm{\phi}\|_{\DG}}\right) .  \label{eq:conditionN2}
\end{multline}
We now bound the second term above. To this end, fix $\bm v , \bm \phi \in \dgg \times \dgl$ and let $\bm w =  \bm \phi - \bm E  \bm \phi$. Define $Z = (f, w)_{\Omega} + (\hat f, \hat w  )_{L_A^2(\Lambda)}  - \mathcal{A}_h (\bm v, \bm w)$.  With the lift operator \eqref{eq:def_lift_operator}, we write 
\begin{align}
Z = (f - L_h v , w)_{\Omega} + ( A\hat f - P\hat{L}_h  \hat{v}  , \hat w)_{\Lambda} - a_h(v, w) - a_{\Lambda, h}(\hat v, \hat w). 
\end{align} 
We integrate by parts the first term in $a_h(v,w)$ and the first term in $a_{\Lambda,h}(\hat v, \hat w)$.  We obtain 
\begin{align}
Z & =  \underbrace{\sum_{K \in \meshg} \int_K (f- L_h v + \Delta v) w + \sum_{i=1}^N \int_{s_{i-1}}^{s_i} (A \hat f - P \hat L_h \hat v +\mathrm{d}_s(A \, \mathrm{d}_s \hat v) ) \hat{w} }_{Z_1} \\ 
\nonumber & \quad \underbrace{- \sum_{F\in \Gamma_h} \int_F [\nabla v] \cdot \bm{n}_F \{w\} - \sum_{i=0}^{N} [A \, \mathrm{d}_s \hat v]_{s_i}\{\hat w\}_{s_i}}_{Z_2} +  Z_3  + Z_4, 
\end{align}
where $Z_3, Z_4$ are the remaining terms in $a_{h}(v,w)$ and $a_{\Lambda,h}(\hat v, \hat w)$ respectively. Namely, 
\begin{align}
Z_3 &= -\epsilon_1 \sum_{F \in \Gamma_h \cup \partial \Omega} \int_F  \{\nabla w\}\cdot \bm{n}_e [v]  - \sum_{F \in \Gamma_h \cup \partial \Omega} \int_F \frac{\sigma_\Omega}{ \vert F\vert^{1/2} } [v][w] , \\ 
Z_4 & =  - \epsilon_2   \sum_{i=1}^{N- 1}  \{A \, \mathrm{d}_s \hat w \}_{s_i} [\hat v]_{s_i} -  \sum_{i=1}^{N - 1}  \frac{\sigma_{\Lambda}}{h_\Lambda}  [\hat v]_{s_i} [\hat w]_{s_i}.  
\end{align}
We start by bounding $Z_3$ and $Z_4$. We note that $[u]= [E\phi] = 0 $ a.e. on $F \in \Gamma_h \cup \partial \Omega$ and that $[\hat u]_{s_i} = [\hat{E}\hat{\phi}]_{s_i} = 0, \,\,\, i \in \{1, \ldots N-1\}$.  We use standard applications of trace inequality for polynomials and Cauchy-Schwarz's inequality  to obtain
\begin{align}
    |Z_3| + |Z_4|  & \lesssim 
 \Vert w\Vert_{\dgg} \Vert v-u\Vert_{\dgg}
+ \Vert \hat{w} \Vert_{\dgl} \Vert \hat{v}-\hat{u}\Vert_{\dgl} \\
& \lesssim 
(\Vert \phi\Vert_{\dgg} + \vert E\phi \vert_{H^1(\Omega)}
+ \vert \hat{\phi}\vert_{\dgl} + \vert\hat{E} \hat{\phi}\vert_{H^1(\Lambda)}) \Vert \bm u -\bm v\Vert_{\DG} \Bk \nonumber \\
&\lesssim \|\bm \phi\|_{\DG} \|\bm u - \bm v\|_{\DG}.  \nonumber
\end{align}
In the last inequality above, we used 
the stability of $\bm E$ given in \eqref{eq:stability_E} of Lemma \ref{lemma:enrich},  and the definition of $\|\cdot \|_{\DG}$. 
For the term $Z_1$, we use Cauchy--Schwarz's inequality and the approximation properties of $\bm E$ \eqref{eq:approx_E}.  We use the notation $\Lambda_i = (s_{i-1}, s_i)$ and we  estimate 
\begin{align}\label{eq:T1}
(Z_1)^2 &\leq  \left( \sum_{K \in \meshg} h_{K}^2 \|f + \Delta v - L_h v\|^2_{L^2(K)} +  \sum_{i=1}^N h_{\Lambda}^2\| A \hat{f} + \mathrm{d}_s(A \, \mathrm{d}_s \hat v) -  P\hat{L}_h \hat v\|^2_{L^2(\Lambda_i)} \right) \\ \nonumber & \quad \quad  \times \left( \sum_{K \in \meshg} h_K^{-2} \|w\|^2_{L^2(K)}  +\sum_{ i = 1}^N  h_{\Lambda}^{-2}\|\hat w\|_{L^2(\Lambda_i)}^2 \right) \\
\nonumber &  \lesssim \left( \sum_{K \in \meshg} h_{K}^2 \|f + \Delta v -  L_h v\|^2_{L^2(K)} +  \sum_{i=1}^N h_{\Lambda}^2\| A \hat{f} + \mathrm{d}_s(A \, \mathrm{d}_s \hat v) -  P\hat{L}_h \hat v\|^2_{L^2(\Lambda_i)} \right)  \|\bm \phi\|^2_{\DG} \\  & : = (R^1_{\Omega} + R^1_\Lambda ) \|\bm \phi\|^2_{\DG}.  \nonumber
\end{align} 
For the term $Z_2$, we use standard applications of trace inequalities and \eqref{eq:approx_E}  to estimate 
\begin{align}\label{eq:T2}
(Z_2)^2  & \lesssim \left( \sum_{F \in \Gamma_h } |F|^{1/2} \|[\nabla v] \cdot \bm{n}_F\|_{L^2(F)}^2 + \sum_{ i = 0 }^N h_{\Lambda} [ A \, \mathrm{d}_s \hat v ]_{s_i}^2\right)  \\ & \quad \quad \times \nonumber   \left( \sum_{K \in \meshg} h_K^{-2} \|w\|^2_{L^2(K)}  +\sum_{i=1}^N h_\Lambda^{-2}\|\hat w\|_{L^2(\Lambda_i)}^2 \right)  \\ 
 \nonumber & \quad \lesssim \left( \sum_{F \in \Gamma_h } |F|^{1/2} \|[\nabla v] \cdot \bm{n}_F \|_{L^2(F)}^2 + \sum_{ i = 0 }^N h_{\Lambda} [A \, \mathrm{d}_s \hat v ]_{s_i}^2\right)   \|\bm \phi\|_{\DG}^2 : = (R^2_\Omega + R^2_\Lambda) \|\bm \phi\|_{\DG}^2 .
\end{align}
 Combining the bounds above we have
\[
\Vert\bm u-\bm u_h\Vert_{\DG} \leq \inf_{\bm v \in \dgg\times \dgl}
\left(\Vert \bm u-\bm v\Vert_{\DG} + (R_\Omega^1+R_\Lambda^1)^{1/2} + (R_\Omega^2+R_\Lambda^2)^{1/2}\right).
\]
The proof is finished by obtaining the required bounds on the residual $(R^1_{\Omega} + R^1_\Lambda )$, see Lemma \ref{lemma:first_aposteriori} and Corollary \ref{cor:first_aposterior}, and on the residual  $(R^2_\Omega + R^2_\Lambda)$, see Lemma \ref{lemma:second_aposteriori} and Corollary \ref{cor:second_aposterior}. 
\end{proof}
\begin{corollary}[Error rate] \label{cor:first_err_estimate} Under the assumptions of \Cref{thm:main_result},  if $u \in H^{3/2-\eta}(\Omega)$  for any $\eta>0$ and $
\hat u \in H^2_A(\Lambda)$, then the following bound holds
\begin{equation}
 \|\bm u - \bm u_h\|_{\DG} \lesssim h^{1/2-\eta} (\|u\|_{H^{3/2-\eta}(\Omega)}+ \|\overline{u} - \hat u \|_{L^2_P(\Lambda)} +  \|f\|_{L^2(\Omega)}) +  h_{\Lambda} (\|\hat u\|_{H^{2}_A(\Lambda)} + \|\hat f\|_{L^2_A(\Lambda)} ). 
\end{equation}
\end{corollary} 
\begin{proof}
Let  $\bm S_h \bm u = (S_h u, \hat S_h \hat u) \in \dgg \times \dgl$ where  $S_h$ and $\hat S_h$ are Scott--Zhang interpolants of $u$ and $\hat u$ respectively \cite{scott1990finite}. With the triangle inequality,  \eqref{eq:trace_estimate_avg}, and approximation properties,  we bound 
\begin{multline}
\|\bm u - \bm S_h \bm u \|_{\DG}  \lesssim \|u - S_h u \|_{H^1(\Omega)} + |\hat u -\hat S_h \hat u |_{H^{1}_A(\Lambda)} + \|\hat u -\hat S_h \hat u \|_{L^2_P(\Lambda)}\\  \lesssim h^{1/2-\eta} \|u\|_{H^{3/2-\eta}(\Omega)}+  h_\Lambda \|\hat u\|_{H^{2}_A(\Lambda)}. 
\end{multline}
Using the above bound in \eqref{eq:main_estimate} yields the desired estimate.
\end{proof} 
We now show that if the mesh is refined near the boundary of $B_{\Lambda}$, (namely the mesh size is of the order $h^{2 k_1}$ where we recall $k_1$ is the polynomial degree for the space $\dgg$) then almost optimal error estimates can be recovered. To this end, we use the definitions of graded meshes \cite{apel1996graded,d2012finite,koppl2018mathematical} in order to obtain the required estimates. 
\begin{corollary}[Graded meshes] \label{cor:second_err_estimate} Let $r_K = \mathrm{dist}(K, \partial B_{\Lambda})$ and  recall that $h_K= \mathrm{diam}(K)$.  Suppose that the mesh satisfies the following grading property. 
\begin{equation} \label{eq:grading}
h_{K} \approx \begin{cases}
    h \, \,  r_K^{1-\frac{1}{2 k_1}}, & \mbox{ if } r_K > \frac12 h_K,  \\ 
    h^{2 k_1},  & \mathrm{otherwise} .  
\end{cases}
\end{equation}
Let $h_{\Lambda} \approx h_{K}, $ for $K \in \meshb$.  Assume that the assumptions of \Cref{thm:main_result} hold. Further,  assume that $u \in H^{k_1+1}(\Omega \backslash \overline{B_{\Lambda}}) \cap H^{k_1+1}( B_{\Lambda})$,  $\hat u \in H^{2}(\Lambda)$, $f \in H^{k_1-1}(\Omega)$, and $\hat{f} \in L^2(\Lambda)$.  Then,
\begin{align} \label{eq:improved_graded}
    \|\bm u - \bm u_h\|_{\DG} \lesssim  & h^{k_1 - 2\eta} (\|u\|_{H^{k_1+1}(B_{\Lambda})} +  \|u\|_{H^{k_1+1}(\Omega \backslash \overline{B_{\Lambda}})}+\|u\|_{H^{3/2-\eta}(\Omega)} + \|f\|_{H^{k_1-1}(\Omega)})   \\ \nonumber &+ h^{2k_1} (\|\hat u \|_{H^{2}(\Lambda)} + \|\hat f \|_{L^2(\Lambda)})  . 
\end{align}
\end{corollary}
\begin{proof}
Define an interpolant $I_h u \in \dgg$ such that $I_h u|_K = S_h u \vert_K$, the Scott--Zhang interpolant restricted to $K$, if $r_{K} \leq \frac12 h_K$ and $I_h u|_K = \pi_h u$, the local $L^2$ projection, otherwise. We use the local approximation properties of the Scott--Zhang interpolant. Namely,  we have that \cite{scott1990finite}
\begin{equation}
|u - S_h u |_{H^1(K)} + h_K^{-1} \|u - S_h u \|_{L^2(K)}   \lesssim h_K^{\min(k_1+1,s) - 1} \|u\|_{H^s(\Delta_K)},  \quad  1 \leq s \leq k_1+1.   
\end{equation}
In the above, $\Delta_K$ is the union of elements sharing a face with $K$. Hence, we obtain that 
\begin{multline}
\sum_{K \in \meshg, r_K \leq \frac12 h_K} \left(  |u - I_h u|^2_{H^1(K)} + h_K^{-2} \|u - I_h u \|^2_{L^2(K)} \right) \\ \lesssim \sum_{K \in \meshg, r_K \leq \frac12 h_K} h_K^{1 - 2\eta}  \|u\|^2_{H^{3/2-\eta}(\Delta_K)}   \lesssim h^{2k_1(1-2\eta)} \|u\|^2_{H^{3/2-\eta}(\Omega)}. 
\end{multline}
Further, using the approximation properties of the $L^2$ projection, we obtain 
\begin{multline}
\sum_{K \in \meshg, r_K >  \frac12 h_K} \left( |u - I_h u |^2_{H^1(K)} + h_K^{-2} \|u - I_h u \|^2_{L^2(K)} \right)  \\ 
\lesssim \sum_{K \in \meshg, r_K >  \frac12 h_K } h^{2k_1} \|u\|^2_{H^{k_1+1}(K) } \lesssim h^{2k_1} (\|u\|^2_{H^{k_1+1}(B_{\Lambda})} +  \|u\|^2_{H^{k_1+1}(\Omega \backslash \overline{B_{\Lambda}})}).   
\end{multline}
In the above, we also used that $r_K \lesssim \diam(\Omega)$. Now, note that 
\begin{align}
\|u -I_h u \|^2_{\dgg} \lesssim \sum_{K \in \meshg } \left(  |u - I_h u |^2_{H^1(K)} + h_K^{-2} \|u - I_h u \|^2_{L^2(K)} \right). 
\end{align}
Define $\bm{I}_h \bm u = (I_h u, \hat S_h \hat u)$. We use the above bounds, triangle inequality, \eqref{eq:boundubar}, and approximation properties of $\hat S_h$ to obtain that 
\begin{align}
\|\bm u &- \bm I_h \bm u\|_{\mathrm{DG}} \lesssim \|u -I_h u \|_{\dgg}  + |\hat u - \hat S_h \hat u |_{\dgl}+ \|\hat u - \hat S_h \hat u\|_{L_P^2(\Lambda)} \\ &\lesssim h^{k_1} (\|u\|_{H^{k_1+1}(B_{\Lambda})} +  \|u\|_{H^{k_1+1}(\Omega \backslash \overline{B_{\Lambda}})} + h^{-2\eta} \|u\|_{H^{3/2-\eta}(\Omega)})  + h^{2k_1} \|\hat u \|_{H^{2} (\Lambda)} . \nonumber  
\end{align}
In the above, we used the assumption that $h_{\Lambda} \approx h_{K}$ for $K \in \meshb$ and thus $h_{\Lambda} \approx h^{2k_1}. $
The above bound estimates the first term of \eqref{eq:main_estimate}. The second and third terms in \eqref{eq:main_estimate} are bounded by the approximation properties of the $L^2$ projections, see \eqref{eq:approximation_l2_projection}. Finally, the last term in \eqref{eq:main_estimate} is controlled by observing that $r_K \lesssim \frac12 h_K, \,\, \forall K \in \mathcal{T}^h_{B}$. Thus, using \eqref{eq:grading}, $h_B \lesssim h^{2k_1}. $ Along with \eqref{eq:trace_estimate_avg}, this concludes the proof.  
\end{proof}
\subsection{Proof details} We now provide details for the steps given in the proof of Theorem \ref{thm:main_result}. \begin{lemma}\label{lemma:condition_N_2}
Let $\bm{u}$ and $\bm{u}_h$ be the solutions of  \eqref{eq:weak_form_grouped} and \eqref{eq:discrete_form_dg} respectively.  Then, \eqref{eq:conditionN2} holds. 
\end{lemma}
\begin{proof}
The proof follows from the abstract framework given in \cite[Lemma 2.1]{gudi2010new}. In the notation of this Lemma, we set $V = H_0^1(\Omega) \times H^1_A(\Lambda)$, $\|\bm{v}\|^2_{V} = \|v\|^2_{H^1(\Omega)} + \|v\|^2_{H^1_A(\Lambda)}$, and $\|\cdot \|_{h} = \|\cdot\|_{\DG}$. We verify assumptions (N1)--(N3) of \cite{gudi2010new}. Observe that assumption (N1) is the coercivity estimate of Lemma \ref{lemma:coercivity}. We now verify assumption (N3) which states that 
\begin{equation}
\|\bm E \bm v\|_{V} \lesssim \|\bm{v}\|_{\DG}, \quad \forall \bm v \in  \dgg \times \dgl.  \label{eq:estimate_N3}
\end{equation}
Let $\bm{v} = (v, \hat v ) \in \dgg \times \dgl$. From Lemma \ref{lemma:enrich}, we have 
\begin{equation}
\|E  v\|_{H^1(\Omega) } + \|\hat E \hat v\|_{H^1_A(\Lambda)} \lesssim \|v\|_{L^2(\Omega)} + \|\hat v\|_{L_A^2(\Lambda)} + \|v\|_{\dgg} + |\hat v|_{\dgl}. 
 \end{equation}
For the first term above, we use Poincar\'e's inequality \eqref{eq:Poincare}. For the second term, we use the fact that $A,P > 0 $, triangle inequality, and trace inequality (Lemma \ref{lemma:trace_for_broken}): 
\begin{equation} \label{eq:poinccare_type_hatv}
\|\hat v\|_{L^2_A(\Lambda)} \lesssim \|\hat v - \overline{v}\|_{L^2_P(\Lambda)} + \|\overline{v}\|_{L^2_P(\Lambda)} \lesssim \|\hat v - \overline{v}\|_{L^2_P(\Lambda)} + \|v\|_{\dgg}.
\end{equation}
Therefore, we obtain that 
\begin{align}
\|E  v\|_{H^1(\Omega) } + \|\hat E \hat v\|_{H^1_A(\Lambda)} \lesssim  \|v\|_{\dgg} + \|\hat v\|_{\dgl} + \|\overline{v} - \hat v\|_{L^2_P(\Lambda)} \lesssim \|\bm{v}\|_{\DG}.
 \end{align}
Hence, \eqref{eq:estimate_N3} is verified.  It remains to verify (N2). We show that for $\bm{v} \in H_0^1(\Omega) \times H^1(\Lambda)$, $\bm{v}_h \in \dgg\times \dgl$, and $\bm{w} \in \cgg \times \cgl $, there holds
\begin{equation}
\mathcal{A}(\bm{v}, \bm{w}) - \mathcal{A}_h( \bm v_h, \bm{w}) \lesssim \|\bm{v} - \bm{v}_h\|_{\DG} (\|w\|_{H^1(\Omega)}^2 + \|\hat w\|_{H^1_A(\Lambda)}^2)^{1/2}. \label{eq:N2}
\end{equation}
For this, observe that $[v]= [w] = 0$ a.e. on $e \in \Gamma_h \cup \partial \Omega$. Thus, we have that
\begin{align*}
a(v,w) - a_h(v_h,w) = \sum_{K \in \meshg} \int_K \nabla (v- v_h) \cdot \nabla w -  \epsilon_1 \sum_{F \in \Gamma_h \cup \partial \Omega}  \int_F \{\nabla w\} \cdot \bm{n}_F [v_h - v]. 
\end{align*} 
With the trace estimate for polynomials  
$$|F|^{1/4} \|\{\nabla w\}\cdot \bm{n}_F\|_{L^2(F)} \lesssim \|\nabla w \|_{L^2(K_F^1 \cup K_F^2)}, \quad F =  \partial K^1_F \cap \partial K_F^2, $$ and  Cauchy-Schwarz's inequality, we obtain that 
\begin{equation}
a(v,w) - a_h(v_h,w) \lesssim \|v_h - v\|_{\dgg} |w|_{H^1(\Omega)}.
\end{equation} 
A similar argument shows that 
\begin{equation}
a_{\Lambda}(\hat v , \hat w) - a_{\Lambda, h}(\hat v_h, \hat{w}) \lesssim |\hat{v}_h - \hat{v}|_{\dgl}| \hat{w}  |_{H^1_A(\Lambda)}.   
\end{equation}
For the remainder terms, we simply use Cauchy-Schwarz's inequality and the trace estimate \eqref{eq:trace_estimate_avg}. Indeed, we have that 
\begin{align*}
b_{\Lambda}(\overline{v}& - \hat v , \overline{w} - \hat w ) - b_{\Lambda} (\overline{v}_h - \hat{v}_h, \overline{w} -\hat w) \leq \xi  \|\overline{v} - \overline{v}_h - ( \hat  v - \hat v_h) \|_{L^2_{P}(\Lambda)} 
\|\overline{w} - \hat w \|_{L^2_P(\Lambda)}\\ & \lesssim   \|\bm{v} - \bm{v}_h\|_{\DG} (\|w\|_{H^1(\Omega)} + \|\hat w \|_{L^2_{P}(\Lambda)}) \lesssim  \|\bm{v} - \bm{v}_h\|_{\DG} (\|w\|_{H^1(\Omega)} + \|\hat w \|_{H^1_A(\Lambda)}). 
\end{align*}
Estimate \eqref{eq:N2} follows  by combining the above bounds. The proof is finished by applying \cite[Lemma 2.1]{gudi2010new}. 
\end{proof}
We now show the first residual bound.
We recall that for any $K\in\mathcal{T}_B^h$, the set
$\mathcal{I}_K$ denotes the set of integers $i_0$  such that $K \in \omega_{i_0}$.
\begin{lemma}[Bound on local residuals over elements] \label{lemma:first_aposteriori}  Fix  $1 \leq i \leq N$ and recall that $\Lambda_i = (s_{i-1},s_i)$.  For all $v_h \in \dgg$ and any $K \in \omega_i$, there holds 
\begin{multline}
 \|f + \Delta v -  L_h v\|^2_{L^2(K)}  \lesssim  h_{K}^{-2} \|\nabla (u-v)\|_{L^2(K)}  + h_{K}^{-1}  \sum_{j \in\mathcal{I}_K}  \| \overline{u} - \hat u \|^2_{L^2_{P}(\Lambda_j)} \\ +  \|L_h (u -v)\|^2_{L^2(K)}  + \|\pi_h f - f\|^2_{L^2(K)}. \label{eq:local_aposteriori_1}
 \end{multline}
For any $\hat v_h \in \dgl$, there holds 
\begin{multline}
 \|A\hat{f} + \mathrm{d}_s( A \, \mathrm{d}_s \hat v) - P \hat{L}_h \hat v\|^2_{L^2(\Lambda_i)} \lesssim  h_{\Lambda}^{-2}\|\mathrm{d}_s (\hat{u} - \hat{v})\|^2_{L_A^2(\Lambda_i)}  + \|\overline{u} - \hat u\|^2_{L^2_{P}(\Lambda_i)}\\  + \|\hat L_h (\hat u - \hat v)  \|^2_{L^2_P(\Lambda_i)} + \|\hat \pi_h \hat f - \hat f\|^2_{L_A^2(\Lambda_i)}. \label{eq:local_aposteriori_2_element}
\end{multline}
\end{lemma}
\begin{proof}
Let $b_K$ be the bubble function associated to $K$ \cite{verfurth1999review}. 
Define the residuals $R =(\pi_h f + \Delta v  - L_h v) \vert_K $ and   $\psi = R \, b_K $. 
Owing to the properties of the bubble functions, we estimate   
\begin{align} \label{eq:expansion}
\|R\|_{L^2(K)}^2 \lesssim \int_{K} R \psi = \int_{K} (f + \Delta v - L_h v) \psi + \int_{K} (\pi_h f - f) \psi= T_1  + T_2.  \nonumber
\end{align} 
Since $\psi$ vanishes on the boundary of $K$, we integrate by parts and obtain 
\begin{align}
T_1  =   \int_{K } (f  \psi - \nabla v  \cdot \nabla \psi - L_h v \, \psi) .  
\end{align}
Testing  \eqref{eq:weak_form_grouped} with $(\psi, 0)$ and substituting in the above gives 
\begin{align}
T_1  = \int_K \nabla (u - v) \cdot \nabla \psi +  b_{\Lambda}(\overline{u} - \hat u, \overline{\psi}) - \int_K L_h v \, \psi . 
\end{align}
The first term is bounded by Cauchy-Schwarz's inequality and inverse estimates since $\psi$ belongs to a finite dimensional space.  
\begin{equation*}
 T_1 \lesssim h_K^{-1} \|\nabla (u-v)\|_{L^2(K)} \|\psi\|_{L^2(K)}+  b_{\Lambda}(\overline{u} - \hat u, \overline{\psi}) - \int_{K} L_h v \, \psi  .     
\end{equation*} 
For the second term above, we use the definition of the lift operator \eqref{eq:def_lift_operator} and write
\[
b_{\Lambda}(\overline{u} - \hat u, \overline{\psi}) = b_{\Lambda}(\overline{u} - \hat u, \overline{\psi} - \overline{\pi_h \psi}) +  (L_h u,  \pi_h\psi)_K = b_{\Lambda}(\overline{u} - \hat u, \overline{\psi} - \overline{\pi_h \psi}) +  (L_h u,  \psi)_K.  
\]
Here we used the definition of the $L^2$ projections in \eqref{eq:local_l2_projection}
and the fact that $L_h u \in \dgg $. 
Since $\psi$ is locally supported on one element $K$, with Cauchy--Schwarz's inequality,  trace estimate \eqref{eq:trace_disc_lambda}, and stability of the $L^2$ projection,  we obtain the bound  
\begin{multline} \nonumber
    \sum_{j\in\mathcal{I}_K} \|\overline \psi - \overline{\pi_h \psi}\|^2_{L^2_P(\Lambda_j)} \leq  \|\psi - \pi_h \psi\|^2_{L^2(\partial B_{\Lambda} \cap K) } \lesssim h_K^{-1} \|\psi - \pi_h \psi\|^2_{L^2(K)} \lesssim   h_K^{-1} \|\psi\|^2_{L^2(K)} . 
\end{multline}
Thus, with Cauchy-Schwarz's and triangle inequalities, we obtain that 
\begin{equation}
 \label{eq:to_bound_blambda}
b_{\Lambda}(\overline{u} - \hat u, \overline{\psi} - \overline{\pi_h \psi}) \lesssim h_K^{-1/2}  (\sum_{j\in\mathcal{I}_K} \|\overline{u} - \hat{u}\|^2_{L^2_P(\Lambda_j)} )^{1/2} \|\psi\|_{L^2(K)}.
\end{equation}
Thus, we obtain 
\begin{align*}
 T_1 \lesssim & 
\Vert \psi\Vert_{L^2(K)} \left( h_K^{-1} \Vert \nabla (u-v)\Vert_{L^2(K)} + 
\Vert L_h(u-v)\Vert_{L^2(K)} + h_K^{-1/2} (\sum_{j\in\mathcal{I}_K} \Vert \bar{u}-\hat{u}\Vert^2_{L_P^2(\Lambda_j)})^{1/2}\right). 
\end{align*} 
The term $T_2$ is simply handled  by Cauchy-Schwarz's inequality.  Collecting the resulting bounds  in \eqref{eq:expansion},  noting that $\|\psi\|_{L^2(K)} 
 \lesssim \|R\|_{L^2(K)}$, and using the triangle inequality  yields estimate \eqref{eq:local_aposteriori_1}. 
 
 To show \eqref{eq:local_aposteriori_2_element}, let $\hat b_i$ denote the bubble functions associated to $\Lambda_i$,  $\hat R = (A\hat{\pi}_h \hat{f}  + \mathrm{d}_s(A \, \mathrm{d}_s  \hat v) - P \hat{L}_h \hat v)\vert_{\Lambda_i}$, and $\hat \psi = \hat R \hat b_i$. We have  
 \begin{align}
 \|\hat R \|_{L^2(\Lambda_i)}^2  &\lesssim  \int_{\Lambda_i} \hat R \hat{\psi}  \nonumber =  \int_{\Lambda_i} (A \hat f +  \mathrm{d}_s(A \, \mathrm{d}_s \hat v) - P \hat L_h \hat v) \hat \psi \nonumber +\int_{\Lambda_i}A(\hat{\pi}_h \hat f - \hat f) \hat{\psi} = T_3 +T_4.  \nonumber
 \end{align} 
 Testing \eqref{eq:weak_form_grouped} with $(0,\hat \psi)$ and performing the same computation as before, we obtain 
 \begin{align}
    T_3 =  \int_{\Lambda_i} A \,\mathrm{d}_s (\hat u - \hat v ) \, \mathrm{d}_s \hat \psi  - b_{\Lambda}(\overline{u} - \hat u,  \hat \psi - \hat \pi_h \hat \psi) + \int_{\Lambda_i} P \hat L_h (\hat u - \hat v) \hat \psi .  
 \end{align}
With Cauchy--Schwarz's and inverse inequalities, the stability of the $L^2$ projection,  and  the fact that $\hat \psi$ is locally supported in $\Lambda_i$, we obtain 
\begin{equation*}
    T_3 \lesssim   \|\hat \psi\|_{L^2_P(\Lambda_i)}(h_{\Lambda}^{-1}\|\mathrm{d}_s (\hat{u} - \hat{v})\|_{L_A^2(\Lambda_i)} \\  + \|\hat L_h (\hat u - \hat v)  \|_{L^2_P(\Lambda_i)}  + \|\overline{u} - \hat u \|_{L^2_P(\Lambda_i)}). 
\end{equation*}
Bounding $T_4$ with Cauchy--Schwarz's inequality and using that $\|\hat \psi\|_{L^2_P(\Lambda_i)} \lesssim \|\hat R\|_{L^2(\Lambda_i)}$,  estimate \eqref{eq:local_aposteriori_2_element} is obtained.  
\end{proof}
An immediate corollary to the above Lemmas is the following global bound. \begin{corollary}\label{cor:first_aposterior}
Recall that $h_{B} = \max_{K \in \meshb} h_K$. The following bound on $R^1_{\Omega} + R^1_{\Lambda}$ (as defined in \eqref{eq:T1}) holds.  
\begin{align} \label{eq:bound_R1}
(R^1_\Omega + R^1_\Lambda) 
\lesssim &   \|\bm u -\bm v \|^2_{ \DG }  + (h_B + h_{\Lambda}^2)  \|\overline{u} - \hat{u}\|_{L^2_P(\Lambda)}^2  
\\ &+  h^2 \|f - \pi_h f \|^2_{L^2(\Omega)} + h_{\Lambda}^2 \|\hat f -\hat \pi_h \hat f \|_{L^2_A(\Lambda)}. \nonumber
  \end{align}
\end{corollary}
\begin{proof}
First note that  if $K \notin \meshb$, then $L_h v = 0 $ on $K$,
and $\mathcal{I}_K = \emptyset$.  We can write
\begin{align*}
R_{\Omega}^1 = 
\sum_{K\in\mathcal{T}_B^h} 
h_K^2 \Vert f+\Delta v-L_h v\Vert_{L^2(K)}^2 
 +\sum_{K\in\meshg\setminus \mathcal{T}_B^h} 
h_K^2 \Vert f+\Delta v\Vert_{L^2(K)}^2.
\end{align*}
For the first term in the right-hand side, we use Lemma~\ref{lemma:first_aposteriori} and the assumption that $|\mathcal{I}_K| \lesssim C$ for all $K \in \meshb$  to obtain the bound:
\begin{align*}
R_{\Omega}^1  \lesssim &\Vert u-v\Vert_{\dgg}^2 
+ h_B \Vert \bar{u} - \hat u \Vert^2_{L_P^2(\Lambda)}
+ h^2 \Vert f - \pi_h f \Vert_{L^2(\Omega)}^2
\\ &+\sum_{K\in\mathcal{T}_B^h} h_K^2 \Vert L_h (u-v)\Vert^2_{L^2(K)} 
+\sum_{K\in\meshg\setminus \mathcal{T}_B^h} 
h_K^2 \Vert f+\Delta v\Vert_{L^2(K)}^2.
\end{align*}
If $K \notin \meshb$, then  standard a posteriori estimates \cite[Section 5.5.1]{di2011mathematical} yield  
\[ h_K^2 \|f+\Delta v\|^2_{L^2(K)} \lesssim \|\nabla (u - v)\|_{L^2(K)}^2 + h_{K}^2  \|f - \pi_h f\|_{L^2(K)}^2, \,\, \forall K \in \meshg \, \backslash \, \meshb. \]
With the bound above and Lemma \ref{lemma:lift_operator}, we can conclude that bound \eqref{eq:bound_R1} holds on $R_{\Omega}^1$. The same bound holds on $R_{\Lambda}^1$ which follows immediately from Lemma~\ref{lemma:first_aposteriori} and Lemma \ref{lemma:lift_operator}. 
\end{proof}

We proceed to bound on $R_\Omega^2 + R_{\Lambda}^2$. For any face $F$, let $S_F = K_F^1\cup K_F^2$ where $K_F^1$ and $K_F^2$ are the elements sharing the face $F$.  We also define
\[
\forall 1\leq i\leq N-1, \quad \hat{S}_i = \Lambda_{i-1}\cup\Lambda_i, \quad
\hat{S}_0 = \Lambda_0, \quad \hat{S}_N = \Lambda_{N-1}, \quad \hat{S}_{N+1} = \emptyset.
\]
\begin{lemma}[Bound on local residual over faces]\label{lemma:second_aposteriori}
Fix $1 
\leq i \leq N$. Then, for any $F \in 
\Gamma_i$,   and any $v \in \dgg$, there holds  
\begin{align}
\label{eq:local_aposteriori_2}
\|[\nabla v] \cdot \bm{n}_F\|^2_{L^2(F)}  \lesssim & \, |F|^{-1/2} \ \sum_{K\subset S_F}  \|\nabla (u-v) \|^2_{L^2(K)}     + |F|^{1/2}  \sum_{K\subset S_F} \|f+ \Delta v -  L_h v\|^2_{L^2(K)} \\
& + |F|^{1/2} \|L_h(u-v)\|_{L^2(S_F)}^2  
 +  \|\overline{u} - \hat{u}\|^2_{L_P^2(  \hat{S}_i \cup \hat{S}_{i+1})}. \nonumber
\end{align}
For any $\hat v \in  \dgl$, there holds  
\begin{align}
\label{eq:local_aposteriori_2b}
  [A\,  \mathrm{d}_s \hat{v}]_{s_i}^2    \lesssim  & \, h_{\Lambda}^{-1} \|\mathrm{d}_s (\hat u- \hat v) \|^2_{L_A^2(\hat{S}_i)}   +  h_{\Lambda}  \sum_{\Lambda_\ell\subset \hat{S}_i} \|A \hat f + \mathrm{d}_s(A \, \mathrm{d}_s \hat v) - P L_h \hat v\|^2_{L^2(\Lambda_\ell)} \\
  & +  h_\Lambda \|\hat L_h (\hat u - \hat v) \|^2_{L^2(\hat{S}_i)} 
 +  h_{\Lambda}\|\overline{u} - \hat{u}\|^2_{L_P^2(\hat{S}_i)}. \nonumber
\end{align}
\end{lemma}
\begin{proof}
Fix $1\leq i\leq N$ and fix $F$ in $\Gamma_i$. Denote by $b_F$ the face bubble associated to $F$; this means that $b_F$ vanishes on the boundary of $S_F$ and $b_F$ takes the value one at the barycenter of $F$. Fix $v$ in $\dgg$. 
  We set $r = [\nabla v] \cdot \bm {n}_F $, extend $r$ by constant values along $\bm{n}_F$,  and set $\psi = r b_F$. 
  From  \cite[proof of Lemma 5.7 (ii)]{di2011mathematical}, we have 
\begin{equation}
\|\psi\|_{L^2(S_F)} \lesssim |F|^{1/4} \|r\|_{L^2(F)}. \label{eq:prop_bubble_faces}
\end{equation}
With the properties of the bubble function and integration by parts, we have 
\begin{align}
\|r\|^2_{L^2(F)} \lesssim  \int_F r \psi = \int_{F} [\nabla v] \cdot \bm{n}_F \psi = \sum_{K\subset S_F} \int_{K} \Delta v \, \psi + \sum_{K\subset S_F} \int_{K} \nabla v \cdot \nabla \psi.  
\end{align}
Choose the test function $\bm{v} = (\psi, 0)$ in \eqref{eq:weak_form_grouped} 
\begin{equation}
\sum_{K\subset S_F}\int_{K} \nabla u \cdot \nabla \psi  +  b_\Lambda(\overline u -\hat u, \overline{\psi})= \int_{S_F}  f \psi. 
\end{equation}
We introduce the $L^2$ projection and rewrite  the second term above as
\begin{align*}
  b_\Lambda(\overline u -\hat u, \overline{\psi})  & =   b_\Lambda(\overline u -\hat u, \overline{\psi} - \overline{\pi_h \psi} ) + (L_h u , \pi_h \psi)_\Omega   
   = b_\Lambda(\overline u -\hat u, \overline{\psi} - \overline{\pi_h \psi}  ) + \int_{S_F} L_h u  \, \psi.  
\end{align*}  
After some manipulation, we obtain 
\begin{align*}
 \|r\|^2_{L^2(F)}  \lesssim  &
\sum_{K\subset S_F} \int_{K} (f + \Delta v -  L_h v) \psi  + \int_{S_F} L_h (v-u) \, \psi \\
& + \sum_{K\subset S_F} \int_{K} \nabla (v-u) \cdot \nabla \psi  
  - b_\Lambda(\overline u -\hat u,  \overline{\psi} - \overline{\pi_h \psi}  )  = W_1 +\ldots + W_4.   \nonumber
\end{align*}
With \eqref{eq:prop_bubble_faces}, the terms $W_1$ and $W_2$ are bounded as:
\begin{align*}
    W_1 +W_2 \lesssim & |F|^{1/4} \Vert r \Vert_{L^2(F)}
    \left( ( \sum_{K\subset S_F} 
    \|f+ \Delta v -  L_h v\|_{L^2(K)}^2)^{1/2}
    +   \|L_h(u-v)\|_{L^2(S_F)}\right). 
\end{align*}
With inverse estimates and \eqref{eq:prop_bubble_faces} and the observation  that $h_{K_F^\ell}^{-1} |F|^{1/4} \lesssim |F|^{-1/4}$ for $\ell = 1,2$, we bound 
\begin{equation*} 
    W_3 \lesssim  |F|^{-1/4} \|r\|_{L^2(F)} (\sum_{K\subset S_F} \|\nabla (u-v) \|_{L^2(K)}^2)^{1/2}.
\end{equation*} 
Let $K_F^1$ and $K_F^2$ denote the elements that share the face $F$
and let  $\mathcal{J}_F$ denote the set of indices $i_0$ such that $K_F^1$ belongs to $\omega_{i_0}$ or such that $K_F^2$ belongs to $\omega_{i_0}$.
In reality, the set $\mathcal{J}_F$ is either the singleton $\{i\}$  (recall that $F$ belongs to $\Gamma_i$) or the pair $\{i, i+1\}$ or the pair
$\{i-1, i\}$ or the triplet  $\{i-1,i,i+1\}$.  
\begin{align*}
W_4 = (\xi P (\overline u-\hat u),\overline{\psi} & -\overline{\pi_h\psi})_\Lambda
=\sum_{\ell\in\mathcal{J}_F} (\xi  P(\overline u-\hat u),\overline{\psi}-\overline{\pi_h\psi})_{\Lambda_\ell}
\\
&\leq \xi (\sum_{\ell\in\mathcal{J}_F} \Vert \overline u-\hat u\Vert_{L^2_P(\Lambda_\ell)}^2)^{1/2}
 (\sum_{\ell\in\mathcal{J}_F} \Vert \overline{\psi}-\overline{\pi_h\psi} \Vert_{L^2_P(\Lambda_\ell)}^2)^{1/2}
\\
&\leq \xi (\sum_{\ell\in\mathcal{J}_F} \Vert \overline u-\hat u\Vert_{L^2_P(\Lambda_\ell)}^2)^{1/2}
\Vert \psi - \pi_h\psi \Vert_{L^2(S_F\cap \partial B_\Lambda) }.
\end{align*} 
With Cauchy-Schwarz inequality and trace estimate \eqref{eq:trace_disc_lambda}, we obtain 
\begin{align*}
  W_4  &\lesssim \xi (\sum_{\ell\in\mathcal{J}_F} \Vert \overline u-\hat u\Vert_{L^2(\Lambda_\ell)}^2)^{1/2}
(\sum_{K\subset S_F} h_K^{-1} \Vert \psi - \pi_h\psi 
\Vert_{L^2(K)}^2)^{1/2}
\\
& \lesssim h_F^{-1/2}
\Vert \overline u-\hat u\Vert_{L^2(\hat S_i\cup\hat S_{i+1})}
\Vert \psi \Vert_{L^2(S_F)},
\end{align*}
where $h_F = \min(h_{K_F^1}, h_{K_F^2})$. 
Therefore, with \eqref{eq:prop_bubble_faces}, we have
\begin{equation}
 W_4 
\lesssim   h_F^{-1/2} \vert F \vert^{1/4} \Vert r \Vert_{L^2(F)}
\Vert \overline u-\hat u\Vert_{L^2(\hat S_i\cup \hat S_{i+1})}
 \lesssim 
\Vert r \Vert_{L^2(F)}
\Vert \overline u-\hat u\Vert_{L^2(\hat S_i\cup\hat S_{i+1})}.
\end{equation} 
Collecting the above bounds and using appropriate Young's inequalities yield \eqref{eq:local_aposteriori_2}.
To prove the bound \eqref{eq:local_aposteriori_2b}, we  denote by $\hat{b}_i$ the typical hat function associated to the node $s_i$; this means that $\hat{b}_i$ is piecewise linear, takes the value $1$ at $s_i$ and the value $0$ at all the other nodes $s_\ell$ for $\ell\neq i$. Denote by $\hat r_i = [A \, \mathrm{d}_s \hat{v}]_{s_i}$ and let $\hat \psi_i = \hat r_i \hat b_i$. It easily follows that
\begin{equation}
 \|\hat{\psi}_i\|_{L^2(\hat S_i)} \lesssim h_{\Lambda}^{1/2} |\hat r_i|. \label{eq:prop_bubble_faces2}
\end{equation}
 Using integration by parts, it is easy to check that
\begin{equation}
    \hat r_i^2 =  [A \, \mathrm{d}_s \hat{v}]_{s_i} \hat \psi_i(s_i) = \sum_{\Lambda_\ell\subset \hat S_i} \left(\int_{\Lambda_\ell} \mathrm{d}_s (A \, \mathrm{d}_s \hat v) \,  \hat \psi_i  + \int_{\Lambda_\ell} A \, \mathrm{d}_s \hat v \, \mathrm{d}_s \hat \psi_i\right).
\end{equation} 
This time, we choose for test function $\bm{v} = (0, \hat{\psi}_i)$ in \eqref{eq:weak_form_grouped} to obtain
\begin{equation}
 \sum_{\Lambda_\ell\subset\hat S_i} \int_{\Lambda_{\ell}} A \, \mathrm{d}_s \hat u \,\, \mathrm{d}_s \hat \psi_i -  b_\Lambda(\overline u -\hat u,  \hat \psi_i )= 
  \int_{\hat S_i} A \hat f \hat \psi_i. 
\end{equation}
We rewrite it as
\begin{equation}
 \sum_{\Lambda_\ell\subset\hat S_i} \int_{\Lambda_{\ell}} A \, \mathrm{d}_s \hat u \,\, \mathrm{d}_s \hat \psi_i -  b_\Lambda(\overline u -\hat u,  \hat \psi_i -\hat \pi_h \hat\psi_i) + \int_{\hat S_i} P \hat L_h \hat u \, \hat \psi_i = 
  \int_{\hat S_i} A \hat f \hat \psi_i. 
\end{equation}
After some manipulation, we obtain 
\begin{align}
  \hat r_i^2 \lesssim &
  \sum_{\Lambda_\ell\subset\hat S_i}\int_{\Lambda_\ell} ( A \hat f + \mathrm{d}_s(A \, \mathrm{d}_s  \hat v) - P \hat L_h \hat v) \hat \psi_i   +  \int_{\hat S_i}  P\hat L_h (\hat v- \hat u) \hat \psi_i   \nonumber \\
& +  \sum_{\Lambda_\ell\subset\hat S_i} \int_{\Lambda_\ell}  A \, \mathrm{d}_s (\hat v - \hat u)\, \mathrm{d}_s \hat \psi_i   + b_\Lambda(\overline u -\hat u,   \hat \psi_i - \hat \pi_h \hat  \psi_i )  = W_5 +\ldots + W_8.   \nonumber
\end{align}
We easily bound the terms $W_5, W_6$ and $W_7$ by \eqref{eq:prop_bubble_faces2}
\begin{align*}
    W_5 +W_6 &\lesssim 
     h_{\Lambda}^{1/2}  |\hat r_i| \left((\sum_{\Lambda_\ell\subset\hat S_i} \| A \hat f + \mathrm{d}_s(A \, \mathrm{d}_s  \hat v) - P \hat L_h \hat v\|_{L^2(\Lambda_\ell)}^2)^{1/2}
    + 
 \|\hat L_h (\hat u - \hat v) \|_{L_P^2(\hat{S}_i)}\right),\\
 W_7 &\lesssim   h_{\Lambda}^{-1/2} |\hat r_i| (\sum_{\Lambda_\ell\subset\hat S_i} \|\mathrm{d}_s (\hat u-\hat v) \|_{L_A^2(\Lambda_\ell)}^2)^{1/2}.  
\end{align*}
For the term $W_8$, we have by Cauchy-Schwarz and stability of the $L^2$--projection that 
\begin{multline*}
W_8  
 \lesssim 
\Vert \overline u-\hat u\Vert_{L^2_P(\hat S_i)}
\Vert \hat \psi_i - \hat \pi_h \hat  \psi_i\Vert_{L_P^2(\hat S_i)}
\\ \lesssim \Vert \overline u-\hat u\Vert_{L_P^2(\hat S_i)}
\Vert \hat \psi_i\Vert_{L^2_P(\hat S_i)}
 \lesssim h_\Lambda^{1/2} \Vert \overline u-\hat u\Vert_{L^2_P(\hat S_i)} \vert \hat r_i\vert.
\end{multline*}
Collecting the above bounds yield the desired result.
\end{proof}
The bound on $(R^2_\Omega + R^2_\Lambda)$ easily follows.  
\begin{corollary}\label{cor:second_aposterior}
 The following bound on $R^2_\Omega + R^2_\Lambda$ as defined in \eqref{eq:T2} holds.
\begin{align} \label{eq:R2_global}
(R^2_\Omega + R^2_\Lambda) 
\lesssim &   \|\bm u -\bm v \|^2_{ \DG }  + (h_B + h_{\Lambda}^2)  \|\overline{u} - \hat{u}\|_{L^2_P(\Lambda)}^2  
\\ &+  h^2 \|f - \pi_h f \|^2_{L^2(\Omega)} + h_{\Lambda}^2 \|\hat f -\hat \pi_h \hat f \|_{L^2_A(\Lambda)}. \nonumber
  \end{align}
\begin{proof}
Recalling the definition of \eqref{eq:def_Gamma_i}, we have 
\[
R_\Omega^2  = \sum_{i=1}^N \sum_{F\in\Gamma_i}
\vert F\vert^{1/2} \Vert [\nabla v]\cdot\bm{n}_F\Vert_{L^2(F)}^2
+ \sum_{F\in\Gamma_h\setminus \bigcup_{i=1}^N \Gamma_i}
\vert F\vert^{1/2} \Vert [\nabla v]\cdot\bm{n}_F\Vert_{L^2(F)}^2.
\]
The first part is bounded using Lemma~\ref{lemma:second_aposteriori}. 
\begin{align*}
 \sum_{i=1}^N \sum_{F\in\Gamma_i}
& \vert F\vert^{1/2} \Vert [\nabla v]\cdot\bm{n}_F\Vert_{L^2(F)}^2 \lesssim \|u-v\|_{\dgg}^2 \\& + \sum_{K \in \meshg} h_K^2 (\|f + \Delta v - L_h v\|^2_{L^2(K)} + \|L_h(u-v)\|^2_{L^2(K)}) + h_B^2 \|\overline{u} - \hat u\|^2_{L^2_{P}(\Lambda)} .
\end{align*}
If $F$ does not belong to $\bigcup_{i=1}^N \Gamma_{i}$, then  $L_h v = 0$ on $S_F$ and  standard a posteriori estimates are used. We omit the details for brevity. Following \cite[Lemma 5.27]{di2011mathematical}, we have 
\begin{equation}
|F|^{1/2}\|[\nabla v]\cdot \bm{n}_F\|^2_{L^2(F)} \lesssim \sum_{K\subset S_F} \|\nabla (u - v)\|_{L^2(K)}^2 + h_{S_F}^2 \|f - \pi_h f \|_{L^2(S_F)}^2 ,\quad \forall F \in \Gamma_h \backslash \bigcup_{i=1}^N \Gamma_{i}  \nonumber
\end{equation}
Combining the above estimates with Lemma \ref{lemma:lift_operator} and Corollary \ref{cor:first_aposterior} yields the bound \eqref{eq:R2_global} on $R_\Omega^2$. For 
$R_{\Lambda}^2$, we  have from Lemma \ref{lemma:second_aposteriori} that 
\begin{align*}
R_\Lambda^2 = \sum_{i=0}^N h_\Lambda [ A \, \mathrm{d}_s \hat v]_{s_i}^2 \lesssim
\sum_{i=0}^N
 \|\mathrm{d}_s (\hat u- \hat v) \|^2_{L_A^2(\hat{S}_i)}  
  +  \sum_{i=0}^N
  h_{\Lambda}^2  \sum_{\Lambda_\ell\subset \hat{S}_i} \|A \hat f + \mathrm{d}_s(A \, \mathrm{d}_s \hat v) - P L_h \hat v\|^2_{L^2(\Lambda_\ell)} \\
  +  h_\Lambda^2    \|\hat L_h (\hat u - \hat v) \|^2_{L_P^2(\Lambda)} 
 +   h_\Lambda^2   \|\overline{u} - \hat{u}\|^2_{L_P^2( \Lambda)}. 
\end{align*}
Applying Corollary~\ref{cor:first_aposterior} and Lemma \ref{lemma:lift_operator} yields the bound on $R_{\Lambda}^2$. 
  \end{proof}
\end{corollary}
\section{Time dependent 3D-1D model }\label{sec:time_dependent}
We now consider the following time dependent model. For further details on the derivation, well--posedness, and regularity properties of the system, we refer to \cite{masri2023modelling}. The weak formulation of the time--dependent problem reads as follows.  Find $ \bm u = (u,\hat u)  \in \bm V = L^2(0,T;H^1_0(\Omega)) \times L^2(0,T; H^1_A(\Lambda))$ with $(\partial_t u, \partial_t \hat u)  \in L^2(0,T;L^2(\Omega)) \times L^2(0,T; L^2_A(\Lambda))$ such that 
\begin{align}
(\partial_t u , v) +  (\partial_t (A\hat u), \hat v)_{\Lambda} + \mathcal{A}(\bm u , \bm  v) & = (f,v) + (A \, \hat f, \hat v)_{\Lambda} , \quad \forall \bm v \in   \bm V.  \\ 
\bm u(0)  & = (u^0, \hat u^0) \in L^2(\Omega) \times L^2_A(\Lambda). 
\end{align}
Here, we recall that $\mathcal{A}$ is given in \eqref{eq:weak_form_grouped} and assume that $f \in L^2(0,T;L^2(\Omega))$ and $ \hat f \in L^2(0,T; L^2_A(\Lambda))$ are given. We retain the assumptions on $A$ and $P$ from the previous sections, and we assume that they are independent of time.  Consider a uniform partition of the time interval $[0,T]$ into $N_T$ sub-intervals with time step size $\tau$. We use the notation $g^n(\cdot) = g(t^n, \cdot) = g(n \tau, \cdot)$ for any function $g$.   Let $(u_h^0, \hat u_h^0) \in \dgg \times \dgl$ be the $L^2$ projection of $(u^0, \hat u^0)$. 
\[
u_h^0 = \pi_h u^0, \quad \hat u_h^0 = \hat \pi_h \hat u^0.
\]
A backward Euler dG approximation then reads as follows. Find $ \bm u_h = (u_h^n ,\hat u_h^n)_{1\leq n \leq N_T} \in  \dgg \times \dgl$ such that 
\begin{multline}
\frac{1}{\tau} (u_h^{n} - u_h^{n-1}, v_h) + \frac{1}{\tau} ( A (\hat u_h^n - \hat u_h^{n-1}), \hat v_h)_{\Lambda} + \mathcal{A}_h(\bm u_h^n, \bm v_h) \\  = (f^n,v_h) + (A \hat f^n, \hat v_h)_{\Lambda} , \quad \forall \bm v_h \in   \dgg \times \dgl. 
\end{multline}
The form $\mathcal{A}_h$ is given in \eqref{eq:dg_form_combined}. To analyse the above scheme, we define the following elliptic projection: $\Pi_h(t): H^1 
 (0, T; H^1_0(\Omega)) \times  H^1(0,T;H^1_A(\Lambda)) \rightarrow H^1(0,T; \dgg ) \times H^1(0,T; \dgl)  $ such that for a given $\bm g(t) = (g(t), \hat g(t))$
\begin{equation*}
    \mathcal{A}_h ( \Pi_h \bm g (t)  , \bm v_h) =  
    (g(t), v_h)_{\Omega} + (A \, \hat{g}(t), \hat v_h)_{\Lambda}. 
\quad \forall \bm{v}_h \in \dgg \times \dgl.
\end{equation*} 
From the analysis of the previous section,  for any $t>0$, $\Pi_h \bm g(t)$ is well defined. Since $\mathcal{A}_h$ is linear and coercive and $\Pi_h$ is continuous, $\partial_t (\Pi_h \bm g(t)) = \Pi_h \partial_t \bm g(t).$  Now, for $\bm u(t) = (u(t), \hat u(t) )$ and $\bm f(t) = (f(t) , \hat f(t))$,  we define the interpolant $\bm \eta_h(t) = (\eta_h (t), \hat \eta_h(t)) \in \dgg \times \dgl$ such that 
$$\bm \eta_h(t)  = \Pi_h (( f(t) - \partial_t u(t) ,\, A \hat f(t) - A\partial_t \hat u(t))).$$
Therefore, we have that 
\begin{equation*}
    \mathcal{A}_h (  \bm \eta_h(t)  , \bm v_h)  =  (f(t)  - \partial_t u (t) , v_h) 
+ (A \, \hat f(t) - A \, \partial_t \hat u(t)  , \hat v_h)_{\Lambda}, \quad \forall \bm v_h \in \dgg \times \dgl.
\end{equation*}
Since $$\mathcal{A}(\bm u(t), \bm v) = (f(t)  - \partial_t u (t) , v)  
+ (A \hat f(t) - A \partial_t \hat u(t)  , \hat v)_{\Lambda}, \quad \forall \bm v \in H^1_0(\Omega) \times H^1_A(\Lambda),$$ we apply the error analysis of the previous section to obtain that for any $\eta>0$  
\begin{multline}  
\|\bm \eta_h(t) - \bm u(t)\|_{\DG} \lesssim h^{1/2 -\eta} (\|u(t)\|_{H^{3/2-\eta}(\Omega)} + \|\hat u(t) \|_{H^2_A(\Lambda)}) \\ + 
h( \|f(t) - \partial_t u(t) \|_{L^2(\Omega)} + \|\hat f (t) - \partial_t u (t) \|_{L^2_A(\Lambda)}) .  \label{eq:error_projection}
\end{multline}
Here, for simplicity, we let $h_{\Lambda} \approx h $. It is also easy to see that
\begin{align*}
    \partial_t \bm{\eta}_h(t) &= \partial_t  \Pi_h (( f(t)-\partial_t u(t),\, A \, \hat f(t) - A\,\partial_t \hat u(t))) \\ & = \Pi_h (( \partial_t f(t)- \partial_{tt} u(t),\, A \, \partial_t \hat f(t) - A\, \partial_{tt} \hat u(t)).  
\end{align*}
Therefore, 
\[ \mathcal{A}_h(\partial_t \bm \eta_h(t) , \bm v_h) =
(\partial_t f(t)  - \partial_{tt} u (t) , v_h)  
+ (A\, \partial_t \hat f(t) -A \, \partial_{tt} \hat u(t)  , \hat v_h)_{\Lambda}, \quad \forall \bm v_h \in \dgg \times \dgl.
\]
Observing that $$\mathcal{A}(\partial_t \bm u(t), \bm v) = (\partial_t f(t)  - \partial_{tt} u (t) , v)  
+ (A\, \partial_t \hat f(t) - A \, \partial_{tt} \hat u(t) , \hat v)_{\Lambda}, \quad \forall \bm v \in H^1_0(\Omega) \times H^1_A(\Lambda),$$ we apply the previous analysis to obtain a bound on $\|\partial_t \bm \eta_h (t) - \partial_t  \bm u(t) \|_{\DG}$ that is a similar to \eqref{eq:error_projection}: 
\begin{multline}  
\|\partial_t  \bm \eta_h(t) - \partial_t \bm u(t)\|_{\DG} \lesssim h^{1/2 -\eta} (\|\partial_t u(t)\|_{H^{3/2-\eta}(\Omega)} + \|\partial_t \hat u(t) \|_{H_A^2(\Lambda)}) \\ + 
h( \|\partial_t f(t) - \partial_{tt} u(t) \|_{L^2(\Omega)} + \|\partial_t \hat f (t) - \partial_{tt}  u (t) \|_{L^2_A(\Lambda)}) .  \label{eq:error_projection_derivative}
\end{multline}
This interpolant allows us to prove the following result. 
\begin{theorem}\label{thm:time_dep_estimate}
For any $1 \leq m \leq N_T$, there holds 
\begin{equation}
\|u_h^m- u^m\|^2 + \|\hat u_h^m - \hat u^m \|_{L^2_{A }(\Lambda)}^2 + \frac{C_{\mathrm{coerc}}}{4}  \tau \sum_{n=1}^{m} \|\bm u_h^n - \bm u^n \|^2_{\DG} \lesssim \tau^2 + h^{1-2\eta} .
\end{equation} 
The above estimate holds under the assumption that $(u,\hat u) \in H^1(0,T; H^{3/2-\eta}(\Omega)) \times H^1(0,T;H^{2}(\Lambda)), \\ (\partial_{tt}u ,\partial_{tt} (A\hat u)) \in L^2(0,T;L^2(\Omega)) \times L^2(0,T;L^2(\Lambda))$, and $(f,\hat f) \in  H^1(0,T;L^2(\Omega)) \times H^1(0,T;L^2(\Lambda))$.
\end{theorem}
\begin{proof}
We  derive the error equation for $\bm e_h^n = (e^n_h, \hat e^n_h) = \bm u_h^n - \bm \eta_h^n.$ For all $\bm v_h \in \dgg \times \dgl$, 
\begin{multline}
\frac{1}{\tau}(e_h^n -e_h^{n-1}, v_h) + \frac{1}{\tau} (A( \hat e_h^n - \hat e_h^{n-1}), \hat v_h)_{\Lambda} +  \mathcal{A}_h  (\bm e^n_h , \bm v_h) \\  = \frac{1}{\tau } (\tau (\partial_t u)^n -(\eta_h^n - \eta_h^{n-1}), v_h ) + \frac{1}{\tau} (\tau A( \partial_t  \hat u)^n - A ( \hat \eta_h^n - \hat \eta_h^{n-1} ), \hat v_h)_{\Lambda}. \label{eq:error_equation_time}
\end{multline}
The proof is based on energy arguments. We  test \eqref{eq:error_equation_time}  with $\bm v_h = \bm e_h^n$ and multiply by $\tau$. 
With the coercivity property \eqref{eq:coercivity}, we obtain 
\begin{align} \label{eq:error_eq_after_testing}
& \frac12 (\|e_h^{n}\|^2 -\|e_h^{n-1}\|^2) + \frac12 (\|\hat e_h^{n}\|_{L^2_{A}(\Lambda)}^2 -\|\hat e_h^{n-1}\|_{L^2_{A}(\Lambda)}^2)  + \frac{C_{\mathrm{coerc}}}{2} \tau \|\bm e_h ^n\|_{\DG}^2 \\ \nonumber  & \lesssim   (\tau (\partial_t u)^n -(\eta_h^n - \eta_h^{n-1}), e_h^n ) \nonumber   + (A( \tau  (\partial_t \hat u)^n - (\hat \eta_h^n - \hat \eta_h^{n-1})) , \hat e_h^{n})_{\Lambda} =T_1 +T_2. 
\end{align} It is standard to show (with Cauchy-Schwarz's inequality, Taylor's theorem, and Poincar\'e's inequality \eqref{eq:Poincare}) that 
\begin{align} \nonumber
T_1  \lesssim (  \tau^{3/2} \|\partial_{tt} u\|_{L^2(t^{n-1}, t^{n}; L^2(\Omega))} + \tau^{1/2} \|\partial_t (u -  \eta_h)\|_{L^2(t^{n-1}, t^{n}; L^2(\Omega))})  \|e_h^n\|_{\dgg} .
\end{align}
With Young's inequality, we then obtain 
\begin{align}
T_1  \leq C    \tau^{2} \|\partial_{tt} u\|_{L^2(t^{n-1}, t^{n}; L^2(\Omega))}^2 + C \|\partial_t (u -  \eta_h )\|_{L^2(t^{n-1}, t^{n}; L^2(\Omega))}^2 +  \tau  \frac{C_{\mathrm{coerc}}}{8}  \|\bm e_h^n\|^2_{\DG}.  \nonumber
\end{align}
Similarly, we bound $T_2$ with 
\begin{align}
T_2 \leq C    \tau^{2} \|A \partial_{tt} 
\hat u\|_{L^2(t^{n-1}, t^{n}; L^2(\Lambda))}^2 + C \|A \partial_t  (\hat u -\hat \eta_h)  \|_{L^2(t^{n-1}, t^{n}; L^2(\Omega))}^2 +  \tau  \frac{C_{\mathrm{coerc}}}{8}  \|\bm e_h^n\|^2_{\DG}.  \nonumber
\end{align}
We use the above bounds in \eqref{eq:error_eq_after_testing}, and we sum the resulting bound over $n$. We obtain that 
\begin{multline} \nonumber
\|e_h^m\|^2 + \|\hat e_h^m \|_{L^2_{A}(\Lambda)}^2 + \frac{C_{\mathrm{coerc}}}{4}  \tau \sum_{n=1}^{m} \|\bm e_h^n\|^2_{\DG} \lesssim \tau^2 ( \|\partial_{tt} u\|_{L^2(0, T;  L^2(\Omega))}^2  + \|A \partial_{tt}   \hat u 
 \|_{0,T; L^2(\Omega))}^2 ) \\ + \|\partial_t (u - \eta_h)\|_{L^2(0, T; L^2(\Omega))}^2   
+ \| A \partial_t  ( \hat u -\hat \eta_h) \|_{L^2(0,T;L^2(\Omega))}^2 + \|e_h^0\|^2 + \|\hat e^0_h\|^2_{L^2_{A}(\Lambda)}. 
\end{multline} 
Then, the result follows by using the error  estimates \eqref{eq:error_projection} and \eqref{eq:error_projection_derivative}, approximation properties of the $L^2$ projections \eqref{eq:approximation_l2_projection}, and the triangle inequality. 
\end{proof}  
\begin{remark}
In the case of graded meshes, i.e. under the same mesh assumptions as Corollary~\ref{cor:second_err_estimate}, almost optimal spatial convergence rates in the dG norm hold. For example, for $k_1 = 1$, we have that for any $1 \leq m \leq N_T$,
\begin{equation}
\|u_h^m- u^m\|^2 + \|\hat u_h^m - \hat u^m \|_{L^2_{A }(\Lambda)}^2 + \frac{C_{\mathrm{coerc}}}{4}  \tau \sum_{n=1}^{m} \|\bm u_h^n - \bm u^n \|^2_{\DG} \lesssim \tau^2 + h^{2(1-2\eta)} .
\end{equation} 
The above estimate holds under the additional assumption that  $u \in H^1(0,T;H^{2}(\Omega \backslash \overline B_{\Lambda}) \cap H^{2}(B_{\Lambda}))$. The proof follows from the same argument as before where similar estimates to \eqref{eq:improved_graded} are used for the dG norm of $\bm \eta_h - \bm u$ and of $\partial_t (\bm \eta_h - \bm u)$. For $k_1 > 1$, one can also derive almost optimal rates under additional regularity requirements on the solution. We omit the details for brevity.  
\end{remark}
\section{Extension to 1D networks embedded in a 3D domain}\label{sec:network}
We extend the above numerical method and model to a 1D network in a 3D domain.  We adopt the notation of \cite{egger2023hybrid} where a hybridized dG method is used for convection diffusion problems in a network. Here, we only introduce Lagrange multipliers on the bifurcation nodes, and we couple the network model to the 3D equations. We do not analyze this dG method for the 3D-1D network model beyond showing well--posedness and local mass conservation at bifurcation points. The error analysis will be the object of future work.

A network is represented by a finite, directed, and connected oriented graph $\mathcal{G}(\mathcal{V}, \mathcal{E})$ where $\mathcal{V}$ is the  set of vertices and $\mathcal{E}$ is the set of edges. 
We let $
\mathcal{E}(\mathsf{v})$ denote the set of edges sharing a vertex $\mathsf{v}$.  The boundary of the graph is then defined by $\mathcal{V}_{\partial} = \{\mathsf{v} \in \mathcal{V}, \,\, \mathrm{card}(\mathcal{E}(\mathsf{v})) = 1\}$. For a given edge $\mathsf e = (\mathsf{v}_\mathrm{in}^{\mathsf e}, \mathsf{v}_{\mathrm{out}}^{\mathsf e})$, we define the function $n_\mathsf e: \mathcal{V} \rightarrow \{-1,0,1\}$ with $$n_\mathsf e(\mathsf{v}_{\mathrm{in}}^\mathsf e) = 1 , \,\, n _\mathsf e(\mathsf{v}_{\mathrm{out}}^ \mathsf e) = -1, \,\, \mathrm{and} \,\,\, n_\mathsf e(\mathsf{v}) = 0, \quad \forall \mathsf{v} \in 
\mathcal{V} \,  \backslash \,  \{ \mathsf{v}_{\mathrm{in}}^\mathsf e,\mathsf{v}_{\mathrm{out}}^\mathsf e \}.  $$
The collection of bifurcation points is denoted by $\mathcal{B} = \{ \mathsf{v} \in \mathcal{V}, \,\,  \mathrm{card}(\mathcal{E}(\mathsf{v})) \geq  3\}$. For each $\mathsf e \in \mathcal{E}$, we define a surrounding cylinder $B_\mathsf e$ of cross--section $
\Theta_\mathsf e$ with area $A_\mathsf e$ and perimeter $P_\mathsf e$.   The $L_P^2$ space over the graph is defined by 
\begin{equation}
L_P^2(\mathcal{G}) = \{u: \,\, u_\mathsf e =  u|_{\mathsf e} \in L^2_{P_\mathsf e}(\mathsf e),  \,\,\, \forall \mathsf e \in \mathcal{E}\}.
\end{equation}

This 1D-network is embedded in a 3D domain $\Omega$. The surrounding cylinders $B_\mathsf e$ are all strictly included in $\Omega$. In $\Omega$, we solve for $u$ satisfying (in the distributional sense) 
\begin{equation} \label{eq:3D_equation_network}
 - \Delta u + \xi (\overline{u} - \hat{ u} ) \delta_{\mathcal G} = f, \quad \mathrm{in} \,\, \Omega, \quad u = 0 \,\, \mathrm{on } \,\, \partial \Omega,
\end{equation}
 and for each $\mathsf e \in \mathcal{E}$, we solve for a 1D solution $\hat u_\mathsf e$ satisfying 
\begin{equation} \label{eq:network_1d_eq}
-\mathrm d_s (A_\mathsf e \, \mathrm{d}_s \hat u_\mathsf e ) + P_\mathsf e(\hat u_\mathsf e - \overline{u}_\mathsf e) = \hat f_\mathsf e \,\, \mathrm{in} \,\,\mathsf e.
\end{equation}
The coefficient $\xi$ is a piecewise positive constant on each edge of the graph. 
The function $\overline{u}$ is defined by
$$ \overline{u}|_\mathsf e  = \overline{u}_\mathsf e = \frac{1}{P_\mathsf e} \int_{\partial \Theta_\mathsf e } u , \quad \forall \mathsf e \in \mathcal{E}.  $$ 
 The functional $ \xi (\overline{u} - \hat{u}) \delta_\mathcal{G}$ is defined  by 
\begin{equation*}
  \xi (\overline{u} - \hat{u})  \delta_\mathcal{G} (v) 
  = \sum_{\mathsf e \in \mathcal{E}} \int_\mathsf e \xi_\mathsf e  \, P_\mathsf e (\overline{u}_\mathsf e - \hat{u}_\mathsf e) \overline{v}_\mathsf e,  \quad \forall v \in H^1(\Omega). 
\end{equation*}
We supplement the above system with the following boundary conditions which impose conservation of fluxes and continuity at bifurcation points. On the boundary, we impose homogeneous Neumann conditions.
\begin{align} \label{eq:b_c_network1}
    \sum_{\mathsf e \in \mathcal{E}(\mathsf{v})} A_\mathsf e \, \mathrm d_s \hat{u}_\mathsf e (\mathsf{v}) n_\mathsf e(\mathsf{v}) & = 0, 
    \,\,\, \mathrm{and} \,\,\, \hat{u}_\mathsf e(\mathsf{v}) = \hat{u}_{\mathsf e'}(\mathsf{v}), 
     && \forall \, \mathsf{v} \in \mathcal{B}, \,\, \forall \, \mathsf e, \mathsf e' \in \mathcal{E}(\mathsf{v}), \\ 
     A_\mathsf e \, \mathrm{d}_s \hat{u}_\mathsf e (\mathsf{v}) & = 0, && \forall \,  \mathsf{v} \in \mathcal{V}_{\partial}, \, \mathsf e \in \mathcal{E}(\mathsf{v}).  \label{eq:bc_network_2}
\end{align} 
To summarise, the 3D-1D network model consists of \eqref{eq:3D_equation_network}-\eqref{eq:network_1d_eq} with boundary conditions \eqref{eq:b_c_network1}-\eqref{eq:bc_network_2}. The above model can also be found in \cite[Section 2.5]{laurino2019derivation}. We now introduce a dG formulation for this model.
\subsection{DG for the 3D-1D network model} For each $\mathsf e \in \mathcal E$,  we denote by $h_\mathsf e$ the length of the edge $\mathsf e$ and  we introduce a mesh and a space $\mathbb V_h^\mathsf e$ of degree $k_\mathsf e$ similar to \eqref{eq:dg_space_1D}. Then, we define the broken polynomial space 
\[ \mathbb{V}_h^{\mathcal{G}} = \{\hat v_h: \,\,\, \hat v_h|_\mathsf e = \hat v_{\mathsf e,h} \in \mathbb V_h^\mathsf e \}. \]
We will use a hybridization technique to handle the values of the discrete solution at the bifurcation points. Thus, we define 
\begin{equation}
\mathbb{V}_h^{\mathcal{B}} = \{ 
\tilde{w}_h = (\tilde{w}_{\mathsf{v},h})_{\mathsf{v}\in\mathcal{B}},  \,\,\, \sum_{\mathsf{v} \in \mathcal{B}} \tilde{w}_{\mathsf{v},h}^2 < \infty \}. 
\end{equation}
 We now define the form $b_\mathsf v: (\mathbb V_h^{\Omega} \times \mathbb{V}_h^{\mathcal{G}} \times \mathbb{V}_h^{\mathcal{B}})^2 \rightarrow \mathbb R$ which enforces conditions at the bifurcation points, see Remark \ref{remark:local_conservation}. For $\mathsf{v} \in \mathcal{B}$, define 
\begin{multline}
    b_{\mathsf{v}}( (u_h, \hat{ u}_h, \tilde{ u}_h), (w_h, \hat{w}_h, \tilde{w}_h)) =   \sum_{\mathsf e \in \mathcal{E}(\mathsf{v})}A_\mathsf e \,  \mathrm d_s \hat{u}_{\mathsf e,h} (\mathsf{v}) n_ \mathsf e(\mathsf{v}) 
    \, (
    \hat{w}_{\mathsf e,h}(\mathsf{v}) - \tilde{w}_{\mathsf{v},h})\\
    +\sum_{\mathsf e \in \mathcal{E}(\mathsf{v})} A_\mathsf e \,  \mathrm d_s \hat{w}_{\mathsf e,h} (\mathsf{v}) n_\mathsf e(\mathsf{v}) 
    \, (
    \hat{u}_{\mathsf e,h}(\mathsf{v}) - \tilde{u}_{\mathsf{v},h})  +  \sum_{\mathsf e \in \mathcal{E}(\mathsf{v})}  \frac{\sigma_{\mathsf{v}}}{h_{\mathsf e}} (\hat{u}_{\mathsf e,h}(\mathsf{v}) - \tilde{u}_{\mathsf{v},h}) \,  (
    \hat{w}_{\mathsf e,h}(\mathsf{v}) - \tilde{w}_{\mathsf{v},h}) .
\end{multline}
The full dG formulation reads as follows. Find $(u_h, \hat{{u}}_h, \tilde{u}_h) \in \mathbb V_h^{\Omega} \times \mathbb{V}_h^{\mathcal{G}} \times \mathbb{V}_h^{\mathcal{B}} $ such that for all $(w_h, \hat{w}_h, \tilde{w}_h) \in \mathbb V_h^{\Omega} \times \mathbb{V}_h^{\mathcal{G}} \times \mathbb{V}_h^{\mathcal{B}}$, there holds
\begin{alignat}{2}
a_h (u_h, w_h) + \sum_{\mathsf e \in \mathcal{E}} b_\mathsf e(\overline{u}_{\mathsf e,h} &  - \hat{u}_{\mathsf e,h}, \overline{w}_{\mathsf e,h}) = (f,w_h), \label{eq:dG_network_0}\\ 
\sum_{\mathsf e \in \mathcal{E}} a_{\mathsf e,h}(\hat{u}_{\mathsf e,h}, \hat{w}_{\mathsf e,h}) + \sum_{\mathsf e \in \mathcal{E}} &b_\mathsf e(\hat{u}_{\mathsf e,h} - \overline{u}_{\mathsf e,h}, \hat{w}_{\mathsf e,h}) \label{eq:dG_network_1} \\ &   + \sum_{\mathsf{v} \in \mathcal{B}} b_{\mathsf{v}}( (u_h, \hat{ u}_h, \tilde{ u}_h), (w_h, \hat{ w}_h, \tilde{ w}_h)) = \sum_{\mathsf e \in \mathcal{E}} (\hat{f}_\mathsf e, \hat{w}_{\mathsf e,h})_{ L_{A_\mathsf e}^2(\mathsf e) }. \nonumber
\end{alignat}
 In the scheme above, the form $a_h$ is the same one defined by \eqref{eq:dg_bilinear_form} and the forms $a_{\mathsf e,h}$ and $b_\mathsf e$ correspond to the forms $a_{\Lambda,h}$ and $b_\Lambda$
with $\Lambda =  \mathsf e$. For instance, we  write
\[
b_\mathsf e(\hat v, \hat w) = (\xi_\mathsf e \hat v, \hat w)_{L_{P_\mathsf e}^2(\mathsf e)}, \quad
\forall \hat v, \hat w \in L_{P_\mathsf e}^2(\mathsf e).
  \]
\begin{remark}[Bifurcation conditions]\label{remark:local_conservation} For a given $\mathsf{v} \in \mathcal{B}$, let $\tilde w_h \in \mathbb V_h^{\mathcal{B}}$ be such that $\tilde w_{\mathsf{v},h}  = 1$ and zero otherwise. Choosing $(w_h, \hat w_h, \tilde w_h) = (0 , 0, \tilde w_h)$ in \eqref{eq:dG_network_1} yields: 
\begin{equation}
\sum_{\mathsf e \in \mathcal{E}(\mathsf{v})} A_\mathsf e \,  \mathrm{d}_s \hat u_{\mathsf e,h}(\mathsf{v}) n_\mathsf e(\mathsf{v}) + \sum_{\mathsf e \in \mathcal{E}(\mathsf{v})}
\frac{\sigma_{\mathsf{v}} }{h_\mathsf e} (\hat u_{\mathsf e,h}(\mathsf{v}) - \tilde u_{\mathsf{v},h})  = 0 , \quad \forall \mathsf{v}\in \mathcal{B}. 
\end{equation}
That is, up to jump terms, the discrete dG scheme locally conserves the fluxes, see \eqref{eq:b_c_network1},  at each bifurcation point. 
 \end{remark}
\begin{lemma}[Well--posedness]
There exists a unique solution for the problem given in \eqref{eq:dG_network_0} and \eqref{eq:dG_network_1}.
\end{lemma}
\begin{proof}
For any $(u_h, \hat u_h , \tilde u_h) \in \mathbb V_h^{\Omega} \times \mathbb{V}_h^{\mathcal{G}} \times \mathbb{V}_h^{\mathcal{B}} $,  let 
\begin{multline}  \nonumber 
    \mathcal{X} = a_h (u_h, u_h) +   \sum_{\mathsf e \in \mathcal{E}} a_{\mathsf e,h}(\hat{u}_{\mathsf e,h}, \hat{u}_{\mathsf e,h}) \\ +  \sum_{e \in \mathcal E} b_\mathsf e (\overline{u}_{\mathsf e,h} - \hat u_{\mathsf e,h} , \overline{u}_{\mathsf e,h} - \hat u_{\mathsf e,h})  +   \sum_{\mathsf{v} \in \mathcal{B}} b_{\mathsf{v}}( (u_h, \hat{ u}_h, \tilde{ u}_h), (u_h, \hat{u}_h, \tilde{u}_h)). 
\end{multline}
It suffices to show that   
\begin{align}\label{eq:coercivity_graph}
  \mathcal{X}  \gtrsim \|u_h\|_{\dgg}^2 + \sum_{\mathsf e \in \mathcal{E}} ( |\hat u_h |^2_{\mathbb{V}_h^{\mathsf e} } + \| \overline{u}_{\mathsf e,h} - \hat{u}_{e,h} \|^2_{L^2_{P}(\mathsf e)}) + \sum_{\mathsf{v} \in \mathcal{B} } \sum_{\mathsf e \in \mathcal{E}(\mathsf{v})}  \frac{\sigma_\mathsf e}{h_{\mathsf e}} (\hat{u}_{\mathsf e,h}(\mathsf v) - \tilde{u}_{\mathsf{v},h})^2,
\end{align}
since the right hand side above defines a norm. Here, $|\cdot |_{\mathbb{V}_h^{\mathsf e}}$ is defined in the same way as \eqref{eq:def_dg_e}. From application of 
trace estimates, it is standard to show that for $\sigma_{\mathsf{v}} $ large enough, there exists a constant $C_3>0$ such that 
    \begin{align*}
\sum_{\mathsf{v} \in \mathcal{B}} b_{\mathsf{v}}( (u_h, \hat{ u}_h, \tilde{ u}_h), (u_h, \hat{ u}_h, \tilde{ u}_h)) 
+ \frac{1}{2}\sum_{\mathsf e \in \mathcal{E}} C_\mathsf e |\hat u_h|^2_{\mathcal{T}_{\mathsf e}^h }\geq  \sum_{\mathsf{v} \in \mathcal{B} } \sum_{\mathsf e \in \mathcal{E}(\mathsf{v})}  \frac{C_3}{h_{\mathsf e}} (\hat{u}_{\mathsf e,h}(\mathsf{v}) - \tilde{u}_{\mathsf{v},h})^2, 
    \end{align*}
where $C_\mathsf e$ is the coercivity constant of $a_{\mathsf e,h}$, similar to \eqref{eq:coercivity_local}. It then follows that 
\begin{align*}
    \sum_{\mathsf e \in \mathcal{E}} a_{\mathsf e,h}(\hat{u}_{\mathsf e,h}, \hat{u}_{ \mathsf e,h}) & +  \sum_{\mathsf{v} \in \mathcal{B}} b_{\mathsf{v}}(u_h, \hat{ u}_h, \tilde{ u}_h), (u_h, \hat{ u}_h, \tilde{ u}_h))  \\ & \geq \frac{1}{2}\sum_{\mathsf e \in \mathcal{E}} C_\mathsf e  |\hat u_h|^2_{\mathcal{T}_{\mathsf e}^h }  +  \sum_{\mathsf{v} \in \mathcal{B} } \sum_{\mathsf e \in \mathcal{E}(\mathsf{v})}  \frac{C_3}{h_{\mathsf e}} (\hat{u}_{\mathsf e,h}(\mathsf{v}) - \tilde{u}_{\mathsf{v},h})^2.
\end{align*}
From here, we use the coercivity results \eqref{eq:coercivity_local}   and the definition of $b_\mathsf e$ to conclude that \eqref{eq:coercivity_graph} holds. We omit the details for brevity.
\end{proof}
\section{Numerical results} \label{sec:numerics}
\subsection{Manufactured solutions with one vessel in a 3D domain} In this first example, we consider manufactured solutions and compute error rates. Let $\Omega = (-0.5,0.5)^3$ contain $\Lambda = \{ (0,0,z), z \in (-0.5,0.5) \}$ with a surrounding cylinder of constant radius $R = 0.05 $. Denoting by $r$ the distance to the line $\Lambda$, the exact solutions are 
\begin{equation}
    u = \begin{cases}
    \frac{\xi}{\xi+1} (1 - R\ln (\frac{r}{R}) ) \hat u ,  & r > R,  \\ 
     \frac{\xi}{\xi+1} \hat u  ,  & r \leq R .
\end{cases} \label{eq:exact_sol},  \quad \mathrm{and} \quad \hat u = \sin(\pi z) + 2. 
\end{equation} 
The above 3D solution is obtained from the observation that \cite[eq. 40]{engquist2005discretization}, see also \cite{koppl2018mathematical}: 
\begin{equation}
\int_{\Omega} - ( \partial_{xx} \,  u + \partial_{yy} u) \,\, v = \int_{\Gamma} \frac{\xi}{\xi + 1} \hat u \, v  = -  \int_{\Lambda} \xi \, P \, (\overline{u} - \hat u ) \overline v .   
\end{equation}
We set $\xi = 1$, and we modify the source terms $f, \hat f$ and the boundary conditions so that the equations are satisfied.
The parameters are set to  $\epsilon_1 = \epsilon_2 = -1$, $k_1=k_2= 1$, and $\sigma_\Omega = \sigma_\Lambda = 30$.
For all our examples, we use the FEniCS finite element framework \cite{alnaes2015fenics,logg2012automated} and the $\mathrm{(FEniCS)}_{ii}$ module \cite{kuchta2020assembly}. 
We compute the solution  
 $(u_h, \hat u_h)$,  the $L^2$ and the $H^1$ norms of the  errors  $e_h = u - u_h$ and $\hat e_h = \hat u - \hat u_h$ on a family of uniform meshes created by FEniCS ``BoxMesh'' with $6N^3$ number of elements. The results and the rates of convergence are reported in Tables \ref{tab:first} and \ref{tab:second} for the 3D and the 1D approximation respectively.
\begin{table}[H]
    \centering
    \begin{tabular}{|c|c|c|c|c|}
       \hline  $N$  & $\|e_h\|_{H^1(\Omega)}$ & rate &  $\|e_h\|_{L^2(\Omega)}$& rate   \\
       \hline \hline
4&  2.313e-01 & -  & 1.562e-02 & -\\  
\hline 
8 & 1.300e-01 & 0.832 & 4.714e-03 &  1.729 \\  \hline 
16 & 8.323e-02 &0.643  & 1.457e-03 &  1.694 \\  \hline 
32 & 5.247e-02 & 0.666& 4.345e-04 &  1.746 \\ \hline 
64 & 3.292e-02 & 0.673  &1.171e-04 &  1.891 \\ \hline 
\end{tabular}
    \caption{$L^2$ and $H^1$ errors and rates between the 3D exact solution \eqref{eq:exact_sol} and the computed solution on a family of  uniform meshes. }
\label{tab:first}
\end{table}

\begin{table}[H]
    \centering
    \begin{tabular}{|c|c|c|c|c|}
       \hline  $N$  & $\|\hat e_h \|_{H^1(\Lambda )}$ & rate &  $\|\hat e_h\|_{L^2(\Lambda)}$& rate   \\
       \hline \hline
4   & 5.008e-01 & -   &3.663e-02 & -  \\ \hline 
 8  & 2.519e-01&  0.992 & 1.779e-02 & 1.042  \\  \hline 
 16 & 1.262e-01 & 0.998  & 7.832e-03 & 1.184 \\ \hline 
 32  & 6.308e-02 &1.000  &  3.374e-03 &1.215 \\  \hline 
 64 & 3.150e-02 & 1.002 & 8.293e-04 & 2.024 \\   \hline 
\end{tabular}
    \caption{$L^2$ and $H^1$ errors and rates between the 1D exact solution \eqref{eq:exact_sol} and the computed solution on a family of  uniform meshes.}
\label{tab:second}
\end{table}
\subsection{Manufactured solution for a vessel network.}  Next, we verify the convergence of the dG scheme for the $1\mathrm{D}$ network model. Precisely, in this example, we now consider only the Poisson problem posed on the network $-\Delta \hat u = \hat f$ on $\mathcal{G}$ complemented  with \eqref{eq:b_c_network1} and   homogeneous
  Dirichlet conditions on $\mathcal{V}_{\partial}$, and we do not solve for a 3D solution. The dG scheme for this 1D diffusion problem problem is given in \eqref{eq:dG_network_1} with $u_h =v_h = \xi = 0$ and the penalty parameters set as $\sigma_\mathsf e= \sigma_\mathsf v =  10$. We consider the network embedded in $\mathbb{R}^2$
  shown in \Cref{tab:hdg_1d} which includes 3 bifurcations, i.e. $\lvert \mathcal{B}  \rvert = 3$,
  located at $\mathsf v_1=(0, 1)$, $ \mathsf v_2= (-1, 2)$, $\mathsf v_3=(1, 2)$ while
  the remaining nodes are placed at $\mathsf v_0=(0, 0)$, 
  $\mathsf v_4=(-1.5, 3)$, $ \mathsf v_5=(-0.5, 3)$, $ \mathsf v_6=(0.5, 3)$, $ \mathsf v_7=(1.5, 3)$. Given $\mathcal{G}$,
  we consider the following  solution and data 
\begin{equation}\label{eq:network_hdg_mms}
    \hat u = \begin{cases}
      y + \cos 2\pi y,  & \!\!(x, y)\in \mathsf e_0\\
      2 + \frac{1}{2}\sqrt{2}(y-1),  & \!\!(x, y) \in \mathsf e_{1\leq i \leq 2}\\
      2 + \frac{1}{2}\sqrt{2} + \frac{1}{8}\sqrt{5}(y-2), & \!\!(x, y)\in \mathsf e_{3\leq i \leq 6}\\            
    \end{cases},
    \,\, 
    \hat f = \begin{cases}
      4\pi^2\cos 2\pi y, & \!\!(x, y)\in \mathsf e_0 \\
      0,                  & \!\!(x, y)\in \mathsf e_{i\neq 0}\\
    \end{cases}.
  \end{equation}    
  Using \eqref{eq:network_hdg_mms} and the dG scheme with linear polynomials,
  \Cref{tab:hdg_1d} confirms the first order convergence of the method.  \begin{table}[H]
    \begin{minipage}{0.45\textwidth}
      \includegraphics[width=0.8\textwidth]{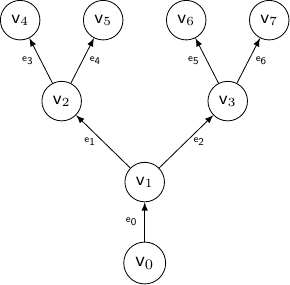}
    \end{minipage}
    \begin{minipage}{0.54\textwidth}
      \scriptsize{
\begin{tabular}{c|cc|cc}
\hline
$h$ &    $\lVert (\hat{e}_h, \tilde{e}_h)\rVert_{\mathbb{V}^{\mathcal{G}}_h\times \mathbb{V}^{\mathcal{B}}_h}$ & rate 
& $\max_{\mathsf{v}\in\mathcal{B}}\lvert j_h(\mathsf{v}) \rvert $ & rate\\[0.5em]
\hline
 5.00e-1 & 1.420 &  --   &  1.242 & -- \\  
 2.50e-1 & 9.705e-1 &  5.49 &  4.683e-1 &  1.41 \\
 1.25e-1 & 4.268e-1 &  1.19 &  1.759e-1 &  1.41 \\
 6.25e-2 & 1.809e-1 &  1.24 &  7.807e-2 &  1.17 \\
 3.13e-2 & 7.896e-2 &  1.18 &  3.770e-2 &  1.05 \\
 1.56e-2 & 3.698e-2 &  1.11 &  1.868e-2 &  1.01 \\
 7.81e-3 & 1.771e-2 &  1.06 &  9.320e-3 &  1.00 \\
 3.91e-3 & 8.656e-3 &  1.03 &  4.657e-3 &  1.00 \\
 1.95e-3 & 4.277e-3 &  1.02 &  2.328e-3 &  1.00 \\
 9.77e-4 & 2.126e-3 &  1.01 &  1.164e-3 &  1.00 \\
 4.88e-4 & 1.060e-3 &  1.00 &  5.821e-4 &  1.00 \\
 2.44e-4 & 5.291e-4 &  1.00 &  2.910e-4 &  1.00 \\
 \hline
\end{tabular}
}
    \end{minipage}
    \caption{Error convergence and flux conservation of the DG scheme defined as part of
      \eqref{eq:dG_network_0}-\eqref{eq:dG_network_1} and applied to the
      standalone diffusion problem on the network shown to the
      left. Here $\hat{e}_h = u - \hat{u}_h$, $\tilde{e}_h = u - \tilde{u}_h$
      with $u$ the exact solution \eqref{eq:network_hdg_mms}. The norm is defined in \eqref{eq:table_norm }. Following
      Remark \ref{remark:local_conservation},  we let
      $j_h(\mathsf{v}) = \sum_{\mathsf e \in \mathcal{E}(\mathsf{v})} \mathrm{d}_s \hat u_{\mathsf e,h}(\mathsf{v}) n_\mathsf e(\mathsf{v})$.
      We set the polynomial degree $k_\mathsf e  = 1$. 
    }
    \label{tab:hdg_1d}
  \end{table}
The norm in \Cref{tab:hdg_1d} is given by 
\begin{equation} \label{eq:table_norm }
\lVert (\hat{e}_h, \tilde{e}_h)\rVert_{\mathbb{V}^{\mathcal{G}}_h\times \mathbb{V}^{\mathcal{B}}_h}^2  = \sum_{\mathsf  e \in \mathcal E} \|\hat{e}_h\|_{\mathbb V_{\mathsf e} ^h}^2 +   \sum_{\mathsf{v} \in \mathcal{B} } \sum_{\mathsf e \in \mathcal{E}(\mathsf{v})}  \frac{\sigma_\mathsf e}{h_{\mathsf e}} (\hat{u}_{\mathsf e,h}(v) - \tilde{u}_{\mathsf{v},h})^2,
\end{equation}
    where $\|\hat{e}_h\|_{\mathbb V_{\mathsf e}^h}^2$ is a slight modification to \eqref{eq:def_dg_e} to also  include boundary terms. 
    
\subsection{Coupled 3D-1D simulation in realistic networks.} In \Cref{fig:silvie}, we finally illustrate the capabilities of our dG scheme to model tissue micro-circulation in a realistic setting. To this end, we utilize
the data set \cite{goirand2021network} which includes vasculature of a
  $1\,\text{mm}^3$ of a mouse cortex, and we  let $\mathcal{G}$ be defined in terms
  of arteries and venules of this network, leaving out the capillaries. The vessel radius in the network ranges approximately from $5\,\mu\text{m}$ to $35\,\mu\text{m}$. 
  The $3$D domain $\Omega$ is then defined as a bounding box of $\mathcal{G}$.
  Upon discretization, $\text{dim}\mathbb{V}^{\Omega}_h=196608$, $\text{dim}\mathbb{V}^{\mathcal{G}}_h=11196$,  and $\text{dim}\mathbb{V}^{\mathcal B}_h=57$.
  The solution fields obtained from \eqref{eq:pde_2} considered with $f=0$, $\hat{f}=1$ and homogeneous Dirichlet and Neumann conditions for $u$ and $\hat{u}$ respectively are shown in \Cref{fig:silvie}.
  \begin{figure}[H]
    \centering
\includegraphics[width=0.5\textwidth]{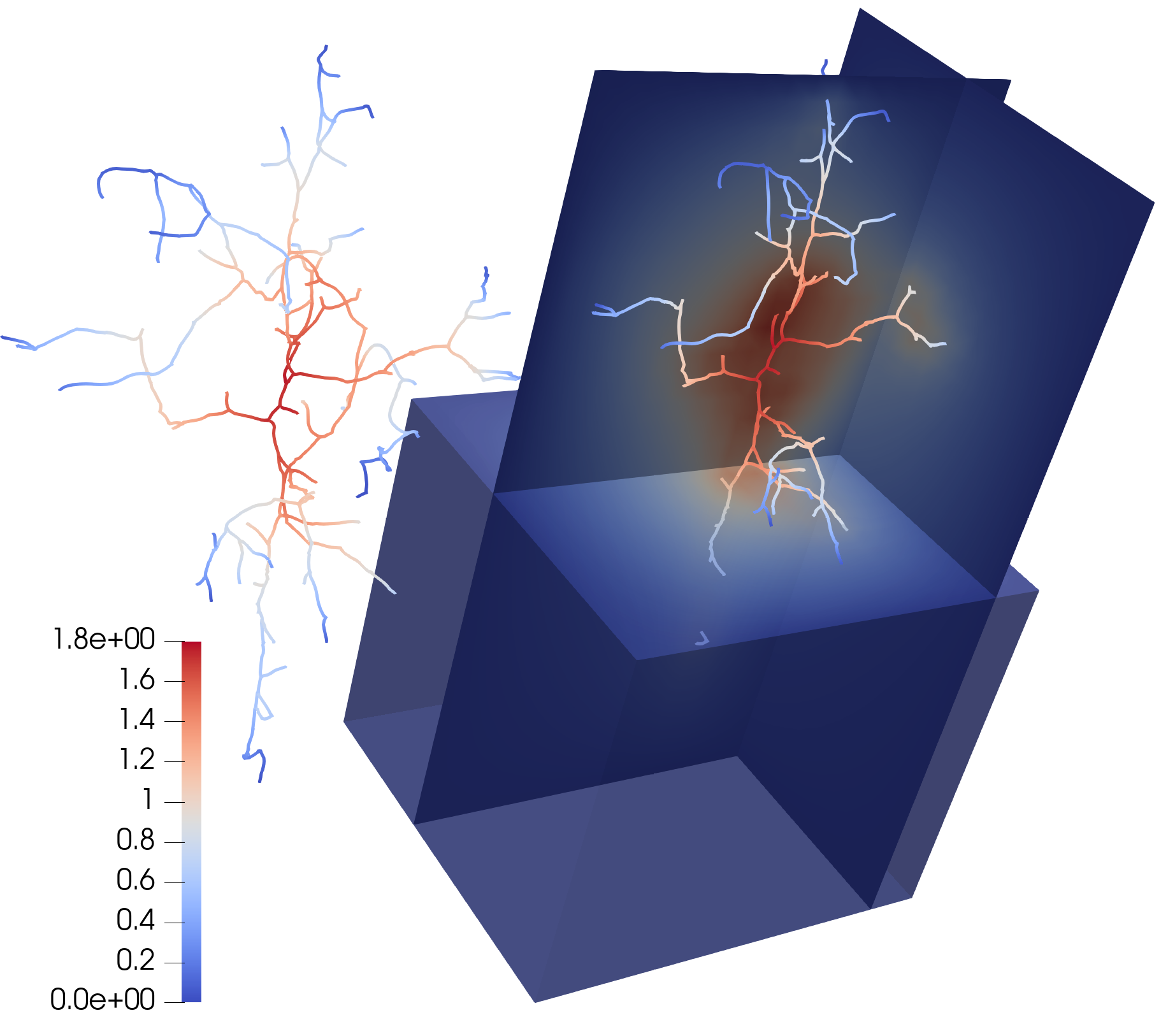}
    \caption{
      Numerical solutions $\hat{u}_h$ and $u_h$ due to the dG scheme 
      \eqref{eq:dG_network_0}-\eqref{eq:dG_network_1}
      applied to 
      the coupled $3$D-$1$D problem \eqref{eq:3D_equation_network}--\eqref{eq:network_1d_eq} with bifurcation conditions \eqref{eq:b_c_network1} 
      considered on a realistic network taken from 
 \cite{goirand2021network}.
    }    
    \label{fig:silvie}
  \end{figure}

\section{Conclusions} \label{sec:conclusion}
 Interior penalty discontinuous Galerkin methods are introduced for coupled $3$D-$1$D problems. These models span several areas of applications such as modeling flow and transport in vascularized tissue. We analyze dG approximations for the steady state problem and a backward Euler dG method for the time dependent problem.  Our analysis is valid under minimal assumptions on the regularity of the solution and on the mesh. Recovering almost optimal rates for graded meshes is also shown, under sufficient regularity assumptions. Further, we propose a novel dG method with hybridization for a network of vessels in a 3D surrounding. The method, up to jump terms, locally conserves mass at bifurcation points. Numerical results demonstrate our error analysis.   
\bibliographystyle{plain}
\bibliography{references}
\end{document}